\documentclass[12pt]{article}

\usepackage[all]{xypic}
\usepackage{amsmath}
\usepackage{amsfonts}
\usepackage{amssymb}
\usepackage{amsthm}
\usepackage{color}

\usepackage[top=2cm, bottom=2cm, left=1cm, right=1cm]{geometry}

\newtheorem{thm}{Theorem}[section]

\newtheorem{crl}[thm]{Corollary}

\newtheorem{prp}[thm]{Proposition}

\newtheorem{lmm}[thm]{Lemma}

\newtheorem{conj}[thm]{Conjecture}

\newtheorem{rmk}[thm]{Remark}

\newcommand {\mb}{\mathbb}
\newcommand {\Z}{\mb Z}
\newcommand {\R}{\mb R}

\newcommand {\colim}{\textrm{colim}\ }

\newcommand {\ex}{\mathrm{excess}}

\begin{document}

\title{Towards a Browder theorem for spherical classes in $\Omega^lS^{l+n}$}

\author{Hadi Zare\\
        School of Mathematics, Statistics,
        and Computer Sciences\\
        College of Science, University of Tehran, Tehran, Iran\\
        \textit{email:hadi.zare} at \textit{ut.ac.ir}}
\date{}

\maketitle

\begin{abstract}
According to Browder \cite{Browder} if $4n+2\neq 2^{t+1}-2$ then the Kervaire invariant of the cobordism class of a $(4n+2)$-dimensional manifold $M^{4n+2}$ vanishes and $M^{2^{t+1}-2}$ is of Kervaire invariant one if and only if $h_t^2\in\mathrm{Ext}_A^{2^t,1}(\Z/2,\Z/2)$ is a permanent cycle. On the other hand, according to Madsen \cite[Theorem 7.3]{Madsenthesis}  if $4n+2\neq 2^t-2$ then $M^{4n+2}$ is cobordant to a sphere (hence of Kervaire invariant zero) and $M^{2^{t+1}-2}$ is not cobordant to a sphere (hence of Kervaire invariant one) if and only if certain element $p_{2^{t}-1}^2\in H_*QS^0$ is spherical. Moreover, it is known that $p_{2^t-1}^2$ is spherical if and only if $h_t^2$ is a permanent cycle in the Adams spectral sequence. Moreover, classes $p_{2n+1}^2\in H_*QS^0$ with $2n+1\neq 2^t-1$ are easily eliminated from being spherical. Hence, Browder's theorem admits a presentation and proof in terms of certain square classes being spherical in $H_*QS^0$ (see also \cite{AkhmetevEccles-Browder} for another geometric proof of Browder's theorem, as well as \cite[Lemmata 4.2 and 4.3]{Eccles-codimension}).\\

Previously, we have determined spherical classes in $H_*\Omega^lS^{n+l}$ for $0<l<4$ and $n>0$ \cite{Zare-Els-1}. In this note, we consider the problem of determining spherical classes $H_*\Omega^lS^{n+l}$ with $n>0$ and $4\leqslant l\leqslant +\infty$ which verifies Eccles conjecture. We start by observing that\\
(1) For $n>0$ and $1\leqslant l\leqslant +\infty$, if $\xi\in H_*\Omega^lS^{n+l}$ is spherical so that $\sigma_*^l\xi=0$ then $\xi=\zeta^2$ for some $A$-annihilated primitive class where $\zeta$ is odd dimensional; here, for $i>0$, $\sigma_*^i:H_*\Omega^lS^{n+l}\to H_{*+i}\Omega^{l-i}S^{n+l}$ is the (iterated) homology suspension. This reduces study of spherical classes to those which are square. We call this the reduction theorem as it reduces the problem to the study of spherical classes that are square. We the proceed to show that\\
(2) For $1\leqslant l<+\infty$, if $\xi^2\in H_*\Omega^lS^{n+l}$ is given with $\dim\xi+1\neq 2^t$ and $\dim \xi+1\equiv 2\textrm{ mod }4$ and $n>l$, then $\xi^2$ is not spherical. We call this as a generalised Browder theorem. Of course, this leaves out the cases with $\dim\xi+1\neq 2^t$ and $\dim\xi+1\equiv0\mathrm{mod}4$. We also present some partial results on the degenerate cases, corresponding to $\dim\xi\neq 2^t-1$, when $l>n$.\\
(3) For $l\in\{4,5,6,7,8\}$ the only spherical classes in $H_*\Omega^lS^{n+l}$ arise from the inclusion of the bottom cell, or the Hopf invariant one elements. Together with our previous work in \cite{Zare-Els-1}, this verifies Eccles conjecture when restricted to finite loop spaces
$\Omega^lS^{n+l}$ with $l<9$ and $n>0$.\\

As an application, we show that\\
(4) If $f\in{_2\pi_{2d}}QS^n$ is given with $h(f)\neq 0$, $\sigma_*h(f)=0$, and $d+1\equiv2\mathrm{mod}4$ then the dimension of the sphere of origin of $f$ is bounded below by $f$, that is if $f$ pulls back to an element of ${_2\pi_{2d}}\Omega^lS^{n+l}$ then $l>n$. This is the first type of this result that we know of in the existing literature providing a lower bounded on the dimension of sphere of origin of an element of ${_2\pi_*^s}$.\\

\textbf{AMS subject classification:$55Q45,55P42$}\\
\textbf{Keywords:} Loop space, James-Hopf map, Dyer-Lashof algebra, Steenrod algebra
\end{abstract}

\tableofcontents

\section{Introduction and statement of results}
For a pointed space $X$, let $QX=\colim\Omega^i\Sigma^iX$ be the infinite loop space associated to $\Sigma^\infty X$. Curtis conjecture on spherical classes in $H_*QS^0$ reads as following.

\begin{conj}(\cite[Theorem 7.1]{Curtis})\label{Curtisconj}
For $n>0$, only Hopf invariant one and Kervaire invariant one elements map nontrivially under the unstable Hurewicz homomorphism
$${_2\pi_n^s}\simeq{_2\pi_n}QS^0\to H_*QS^0.$$
\end{conj}

Throughout the paper, we shall work at the prime $2$, the homology will be $\Z/2$-homology; we write ${_2\pi_*^s}$ and ${_2\pi_*}$ for the $2$-component of $\pi_*^s$ and $\pi_*$ respectively. We also write $H_*$ for $H_*(-;\Z/2)$.\\

A more generalised question, is the determination of the image of the Hurewicz homomorphism $\pi_*\Omega^\infty E\to H_*\Omega^\infty E$ where $E$ is a suitable spectrum. The aim of this paper is to continue this investigation when $E=\Sigma^\infty S^n$ with $n>0$ by means of examining the conjecture on finite loops spaces associated to spheres. Let us recall a variant of Curtis conjecture, due to Eccles, which may be stated as follows.

\begin{conj}\label{Ecclesconj}
(\textrm{Eccles conjecture}) Let $X$ be a path connected $CW$-complex with finitely generated homology. For $n>0$, suppose $h(f)\neq 0$ where ${_2\pi_n^s}X\simeq{_2\pi_n}QX\to H_*QX$ is the unstable Hurewicz homomorphism. Then, the stable adjoint of $f$ either is detected by homology or is detected by a primary operation in its mapping cone.
\end{conj}

The space $QX=\colim\Omega^l\Sigma^lX$, by definition, is filtered by spaces $\Omega^i\Sigma^iX$. So, it is natural to look for spherical classes in $H_*\Omega^l\Sigma^lX$ for $l>0$. Our first observation is that if $f:S^n\to QX$ is given with $h(f)=0$ so that $h(f)$ is not stably spherical then for $X=S^n$ with $n>0$ as well as many favorable cases, it is enough to eliminate the cases with $h(f)$ being a square. We have the following which we call the Reduction theorem.

\begin{thm}\label{TheReductionTheorem}
(i) Suppose $0<l\leqslant +\infty$ and $n>0$. If $f:S^i\to\Omega^l S^{n+l}$ is given with $h(f)\neq 0$ and $\sigma_*^lh(f)=0$. Then, there exist some $j<+\infty$ so that for the adjoint map $f^j:S^{i+j}\to\Omega^{l-j}S^{(n+j)+l}$ we have
$$h(f^j)=(\sum Q^Ix_{n+j})^2$$
where $\sum Q^Ix_{n+j}$ is odd dimensional.\\
(ii) Suppose $X=\Sigma^2Z$ for some space $Z$. Suppose $f:S^i\to QX$ is given such that $h(f)\neq 0$ and $\sigma^\infty h(f)=0$. Then, there exist some $j<+\infty$ so that for the adjoint map $f^j:S^{i+j}\to Q\Sigma^jX$ we have
$$h(f^j)=(\sum Q^I\Sigma \zeta_\alpha)^2$$
where $\sum Q^I\zeta_\alpha$ is odd dimensional, and $\{\zeta_\alpha\}$ is a homogeneous basis for the reduced homology $\widetilde{H}_*X$.
\end{thm}

By the above theorem, in order to classify spherical classes $\xi$ in $H_*\Omega^lS^{n+l}$ or $H_*QX$ in terms elements $f\in\pi_*\Omega^lS^{n+l}$ or $f\in\pi_*^sX$ with $h(f)=\xi$ that are not stably spherical, it is enough to consider spherical classes which are square. Our next main observation is a kind of generalised Browder theorem. As for terminology, we refer to the cases $\xi^2$ with $\dim\xi\neq 2^t-1$ for all $t>0$ as the degenerate cases, opposite to the nondegenerate cases with $\dim\xi=2^t-1$ for some $t>0$. The following provides a case in which many degenerate cases are eliminated.

\begin{lmm}\label{main1}
Suppose $X$ is path connected with $H^{2d-1}X\simeq 0$, $H^{2d}X\simeq 0$, and $f:S^{2d}\to QX$ is given with $d$ an odd number such that $d+1\neq 2^t$ for all $t>0$, and $d+1\equiv 2\textrm{ mod }4$. Then, is it impossible to have $h(f)=\xi^2$ for some $\xi\in H_*X$.
\end{lmm}

The proof is quite easy, so we include it here.

\begin{proof}
Suppose there exists $f$ with $h(f)=\xi^2$ as above. By \cite[Proposition 5.8]{AsadiEccles} the stable adjoint of $f$, say $f^s:S^{2d+1}\not\to X$, is detected by $Sq^{d+1}$ on a $d+1$ dimensional class $\xi^*\in H^*(X\cup_{f^s}e^{2d+1})$. For $d+1=4m+2$ for some $m$, we have the Adem relation
(see for instance \cite[Page 31]{MosherTangora})
$$Sq^{d+1}=Sq^{4m+2}=Sq^2Sq^{4m}+Sq^1Sq^{4m}Sq^1.$$
Now, $Sq^{4m}\xi^*$ and $Sq^{4m}Sq^1\xi^*$, respectively, live in $H^{2d}C_{f^s}\simeq H^{2d}X\simeq 0$ and $H^{2d-1}C_{f^s}\simeq H^{2d-1}X\simeq 0$. Therefore, $Sq^{d+1}\xi^*=0$ in $H^*C_{f^s}$ which is a contradiction.
\end{proof}

Next, note that for $f:S^d\to\Omega^lS^{n+l}$ with $n>0$, assuming $h(f)\neq 0$, we may write
$$h(f)=\sum \epsilon_JQ_Jx_n$$
where $\epsilon\in\Z/2$ and $J$ is running over all increasing sequence with $\dim Q_Jx_n=d$. Our next tool is provided by the following observation.

\begin{thm}\label{detect1}
Suppose $f:S^{2d}\to\Omega^lS^{n+l}$ is given with $h(f)=(\sum Q_Jx_n)^2$. Then, for $l_0=\max\{l(J)\}$, the composition
$$(\Omega j_{2^{l_0}}^{l-1})\circ f:S^{2d}\to\Omega^lS^{n+l}\to\Omega QD_{2^{l_0}}(S^{n+1},l-1)$$
satisfies
$$h(\Omega j_{2^{l_0}}^{l-1}f)=(\sum_{l(J)=l_0} \Sigma^{-1}\overline{Q_Jx_{n+1}})^2$$
and its stable adjoint, say $g:S^{2d+1}\not\to D_{2^{l_0}}(S^{n+1},l-1)$, is detected by $Sq^{d+1}$ on a $(d+1)$-dimensional class in its stable mapping cone.
\end{thm}

Next, as an application, we record an observation on the degenerate cases which seems to be first result of this type appearing in this context.

\begin{thm}\label{degenerate}
Suppose $d\neq 2^t-1$ for all $t>0$, $l>3$, $d+1\equiv2\mathrm{mod}4$, and $f:S^{2d}\to\Omega^{l}S^{n+l}$ is given with $h(f)=(\sum Q_Jx_n)^2$. Then, $l>n$. Moreover, in the case $l>n$ we have $l(J)\leqslant l-n-\frac{l}{2^{l_0}}$ for all $J$ involved in $h(f)$ where $l_0=\max\{l(J)\}$. Consequently, if $f\in\pi_{2d}QS^n$ is given with $h(f)\neq 0$ and $\sigma_*h(f)=0$ then $2d\geqslant 3n$. Furthermore, there exists $l>n$ so that $f$ admits a pull back to $\pi_{2d}\Omega^lS^{n+l}$.
\end{thm}

We also apply the above observations to the table provided in \cite[Lemma 8.1]{Zare-Els-1} together with some detailed analysis in order to verify Eccles conjecture on $H_*\Omega^lS^{n+l}$ with $l<9$, we prove the following.

\begin{thm}\label{main2}
For $3<l<9$ and $n>0$, the only spherical classes in $H_*\Omega^lS^{n+l}$ arise either from the inclusion of the bottom cell, or the Hopf invariant one elements.
\end{thm}

Together with results of \cite{Zare-Els-1} this completes the determination of spherical classes in $H_*\Omega^lS^{n+l}$ for $l<9$. Our computations have immediate applications to the theory of bordism class of immersions and their (stable) characteristic classes, similar to \cite{Zare-filteredfiniteness}. We leave further discussion on this to a future work.

\section{Recollections}

\subsection{Homology suspension}
For a space $X$, the adjoint of the identity map $\Omega X\to\Omega X$ is the evaluation map $e:\Sigma\Omega X\to X$ which in homology induces the homology suspension $\sigma_*:H_*\Omega X\to H_{*+1}X$. For an iterated loop space $\Omega^lX\to\Omega^lX$, the iterated application of the above procedure yields the iterated homology suspension $\sigma_*^k:H_{*-k}\Omega^{l}X\to H_*\Omega^{l-k}X$ where $k\leqslant l$. We allow $l=+\infty$ where the notation $\Omega^\infty E=\colim E_i$ means the infinite loop space associated to a spectrum $E$ with underlying space $E_i$ and the structure maps $E_i\to\Omega E_{i+1}$ are used to define the colimit. In this case, we also have the stable homology suspension $\sigma_*^\infty:H_*\Omega^\infty E\to H_*E$. Note that if $E=\Sigma^\infty X$ for some space $X$ then $H_*\Sigma^\infty X\simeq \widetilde{H}_*X$ and the stable homology suspension reads as $\sigma_*^\infty:H_*\Omega^\infty\Sigma^\infty X\to \widetilde{H}_*X$.

\subsection{Homology of iterated loop spaces}\label{homologysection}
It is known that homology of spaces $QS^n$ and $\Omega^iS^{i+n}$, with $n\geqslant 0$, can be described using Kudo-Araki operation; the Browder operations and Cohen brackets vanish on $\Omega^iS^{i+n}$ when $n<+\infty$ (see for example \cite[Theorem 3]{KudoAraki},\cite[Theorem 7.1]{KudoAraki-Hn},
\cite[Page 86, Corollary 2]{DyerLashof}, \cite{CLM}). We wish to recall the descriptions in both lower indexed, and upper indexed operations; both of these descriptions are useful in proving some of our statements. \\

\textbf{Homology in terms of $Q_j$ operations.} For an $i$-fold loop space $X$, the operation $Q_j$ is defined for $0\leqslant j<i$ as an additive homomorphism
$$Q_j:H_*X\to H_{2*+j}X$$
so that $Q_0$ is the same as squaring with respect to the Pontrjagin product on $H_*X$ coming from the loop sum on $X$. The homology rings $H_*\Omega^i\Sigma^i S^n$ and $H_*QS^n$ when $n>0$, as algebras, can be described as
$$\begin{array}{lll}
H_*\Omega^i\Sigma^iS^n &\simeq &\Z/2[Q_{j_1}\cdots Q_{j_r}x_n:0<j_1\leqslant j_2\leqslant\cdots\leqslant j_r<i],\\
H_*QS^n                &\simeq &\Z/2[Q_{j_1}\cdots Q_{j_r}x_n:0<j_1\leqslant j_2\leqslant\cdots\leqslant j_r],
\end{array}$$
where, for $n>0$, $x_n\in\widetilde{H}_nS^n$ is a generator. Note that $Q_0Q_Jx_n=(Q_Jx_n)^2$ in the polynomial algebra, so this case is not included in the above description. In this description, we allow the empty sequence $\phi$ as a nondecreasing sequence of nonnegative integers with $Q_\phi$ acting as the identity; this realises the monomorphism $H_*S^n\to H_*\Omega^i\Sigma^iS^n$ given by the inclusion of the bottom cell $S^n\to\Omega^i\Sigma^iS^n$ being adjoint to the identity $S^{i+n}\to S^{i+n}$ which sends $x\in H_*S^n$ to $Q^\phi x$.\\

\textbf{Homology in terms of $Q^i$ operations.} For some purposes, it is easier to use description of $H_*QX$ in terms of the operations $Q^i$. For an infinite loop space $Y$ these are additive homomorphisms $Q^i:H_*Y\to H_{*+i}Y$ operations which relate to lower indexed operations by $Q_iz=Q^{i+d}z$ for any $d$-dimensional homology class $z$. The description of homology in lower indexed operations, translates to the following
$$\begin{array}{lll}
H_*QS^n    &\simeq&\Z/2[Q^Ix_n:I\textrm{ admissible },\ex(I)>n],\\
H_*Q_0S^0  &\simeq&\Z/2[Q^I[1]*[-2^{l(I)}]:I\textrm{ admissible }],
\end{array}$$
where $I=(i_1,\ldots,i_s)$ is admissible if $i_j\leq 2i_{j+1}$, and $\ex(I)=i_1-(i_2+\cdots+i_s)$; in particular, note that for $Q^Ix_n=Q_Jx_n$, $\ex(I)=j_1$ and $I$ is admissible if and only if $J$ is nondecreasing. We allow the empty sequence to be admissible with $\ex(\phi)=+\infty$ and $Q^\phi$ acting as the identity. In particular, in this description, $Q^iz=0$ if $i<\dim z$ and $Q^dz=z^2$ if $d=\dim z$.\\
The evaluation map $\Sigma QS^n\to QS^{n+1}$, adjoint to the identity $QS^n\to \Omega QS^{n+1}=QS^n$, induces homology suspension $H_*QS^n\to H_{*+1}QS^{n+1}$. According to \cite[Page 47]{CLM} the homology suspension is characterised by the following properties: (1) $\sigma_*$ acts trivially on decomposable terms; (2) on the generators it is given by
$$\sigma_* Q^Ix_n=Q^Ix_{n+1}\textrm{ if }n>0,\ \sigma_*(Q^I[1]*[-2^{l(I)}])=Q^Ix_1\textrm{ if }n=0.$$
The action of the Steenrod algebra on $H_*Q_0S^n$, $n\geqslant 0$, is determined by iterated application of Nishida relations which read as follows
$$Sq_*^rQ^a=\sum_{t\geqslant 0}{a-r\choose r-2t}Q^{a-r+t}Sq^{t}_*$$
Here, $Sq^r_*:H_*(-;\Z/2)\to H_{*-r}(-;\Z/2)$ is the operation induced by $Sq^r$ using the duality of vector spaces over $\Z/2$. Note that the {least upper bound for $t$ so that the binomial coefficient could be nontrivial mod $2$, is $[r/2]+1$ (depending on the parity of $r$ this would be maximum value or maximum value of $t$ plus $1$)}. In particular, we have
$$Sq^1_*Q^{2d}=Q^{2d-1},\ Sq^{1}_*Q^{2t+1}=0.$$
We also have the Cartan formula $Sq^{2t}_*\zeta^2=(Sq^t_*\zeta)^2$ \cite{Wellington}.\\

\textbf{Homology of QX for $X$ path connected.} Suppose $\{x_\alpha\}$ is an additive basis for the reduced $\Z/2$-homology $\widetilde{H}_*X$. Then, as an algebra over the Dyer-Lashof algebra, we have
$$H_*QX\simeq\Z/2[Q^Ix_\alpha:I\textrm{ admissible},\ex(I)>\dim x_\alpha]$$
where the notions of admissibility and excess retain the same meanings as above. We note that, at least when $X$ is path connected, $H_*QX$ as an algebra over $R$ can be described as the free commutative algebra generated by symbols $Q^Ix_\alpha$ with $I$ admissible subject to the following relations
$$\ex(I)<\dim x_\alpha\Rightarrow Q^Ix_\alpha=0, Q^nx_\alpha=x_n^2\Leftrightarrow \dim x_\alpha=n.$$
The homology suspension $\sigma_*:H_*QX\to H_*Q\Sigma X$ induced by the evaluation map $\Sigma QX\to Q\Sigma X$ kills decomposable elements and on the generators is described by \cite[Part I, Page 47]{CLM}
$$\sigma_*Q^Ix_\alpha=Q^I\Sigma x_\alpha.$$
The following is immediate from the description of $\sigma_*$ and $H_*QX$ as well as $Q_*\Omega^i\Sigma^iS^n$.
\begin{lmm}\label{kernelofsuspension}
(i) For $0<i\leqslant +\infty$ and $n\geqslant 0$, the kernel of $\sigma_*:H_*\Omega^i_0\Sigma^iS^n\to H_{*+1}\Omega^{i-1}\Sigma^{i-1}S^{n+1}$ is the subalgebra generated by all decomposable elements (we use the convention $+\infty-1=+\infty$.\\
(ii) For $X$ path connected, the kernel of $\sigma_*:H_*QX\to H_*Q\Sigma X$ is the subalgebra generated by all decomposable elements.
\end{lmm}

\begin{proof}
Case (i) is \cite[Lemma 3.4]{Zare-Els-1}. The statement of (ii) follows analogous to part (i) as in \cite{Zare-Els-1}. Assume $\xi=\sum Q^Ix_\alpha +D$ where $D$ is a sum of decomposable elements. If the first term is a sum of terms with $\ex(I)>\dim x_\alpha$ then $\sigma_*\xi=\sum Q^I\Sigma x_\alpha\neq 0$. This proves our claim.
\end{proof}

\subsection{Spherical classes}
A  class $\xi\in\widetilde{H}_*X$ is called spherical if it is in the image of the Hurewicz homomorphism $h:\pi_*X\to H_*X$. A spherical class has some basic properties: (1) it is primitive in the coalgebra $\widetilde{H}_*X$ {where the coalgebra structure is induced by the diagonal map $X\to X\times X$}, (2) $Sq^i_*\xi=0$ for all $i>0$ where $Sq^i_*:\widetilde{H}_*X\to \widetilde{H}_{*-i}X$ is the operation induced by $Sq^i:\widetilde{H}^*X\to \widetilde{H}^{*+i}X$ by the vector space duality \cite[Lemma 6.2]{AsadiEccles}; for simplicity, we say $\xi$ is $A$-annihilated if $Sq^t_*\xi=0$ for all $t>0$. There is a spherical class $\xi_{-1}\in H_{*-1}\Omega X$, not necessarily unique, so that $\sigma_*\xi_{-1}=\xi$.\\

For a spectrum $E$, we say $\xi\in H_*\Omega^\infty E$ is stably spherical if $\sigma_*^\infty\xi\neq 0$. Writing $h^s:\pi_*E\to H_*E$ for the stable Hurewicz map, for $f\in \pi_i\Omega^\infty E$, $h(f)$ is stable spherical if and only if the stable adjoint of $f$, say $f^s:S^i\not\to E$ maps nontrivially under $h^s$. In particular, if $\xi\in H_{i}QS^n$ is stably spherical then $\sigma_*^\infty\xi\neq 0$ in $\widetilde{H}_iS^n$ which shows that $i=n$. Similarly, if $\xi\in H_i\Omega^lS^{n+l}$ is spherical with $\sigma_*^l\xi\neq 0$ then $i=n$. The following now is immediate.
\begin{prp}\label{stablyspherical-1}
Suppose $0<l\leqslant +\infty$. Suppose $f\in\pi_i\Omega^l\Sigma^lS^{n}$ so that $\sigma_*^lh(f)\neq 0$ where $h:\pi_i\Omega^l\Sigma^lS^{n}\to H_*(\Omega^l\Sigma^lS^n;\Z)$ is the integral Hurewicz homomorphism. Then $i=n$ and the adjoint of $f$ as an element of $\pi_{i+n}S^{i+n}$ is detected by its homological degree. Similarly, at the prime $2$, if $f\in{_2\pi_i}\Omega^l\Sigma^lS^{n}$ is given with $\sigma_*^lh(f)\neq 0$ then $i=n$ and the adjoint of $f$ as an element of ${_2\pi_{i+n}}S^{i+n}$ is detected by mod $2$ degree.
\end{prp}

\subsection{The Reduction theorem: Proof of Theorem \ref{TheReductionTheorem}}
This observation is probably known to experts in the field, rather as a triviality, but we don't know of any published account. We record this as besides its trivial looking face, it has immediate application in eliminating many classes from being spherical. For instance, note that a spherical class is primitive which if decomposable then by \cite[Proposition 4.23]{MM} it must be square of a primitive class, and possibly higher powers of $2$ of some primitive class. The reduction theorem eliminates higher powers of $2$. Let is recall the following observations.

\begin{lmm}
(\cite[Lemma 3.6]{Zare-Els-1}) If $\xi\in H_m\Omega^lS^{n+l}$, $0<l\leqslant +\infty$ and $n>0$, is a decomposable spherical class, then $\xi=\zeta^2$ for some odd dimensional primitive class $\zeta$.
\end{lmm}

The case of $n=1$ of the above theorem contains a gap in proof which later on was fixed in the following.

\begin{thm}
(\cite[Theorem 1.8]{Zare-filteredfiniteness}) (i) Suppose $\zeta^2\in H_*QS^1$ is a spherical class, then $\zeta$ is odd dimensional.\\
(ii) Suppose $X$ is an arbitrary space of finite type and $\zeta^2\in H_*Q\Sigma^2X$ is spherical, then $\zeta$ is odd dimensional.
\end{thm}

The fix comes from the fact that the stabilisation map $\Omega^lS^{l+1}\to QS^1$ induces a multiplicative monomorphism in homology and an application of the above theorem. The proof of the reduction theorem is now obvious, but we outline a proof.

\begin{proof}[Proof of Theorem \ref{TheReductionTheorem}]
Suppose $f:S^n\to Y$, $Y=Q\Sigma^2X$ or $Y=\Omega^lS^{n+l}$ with $l$ and $n$ chosen as above, is given with $\sigma^\infty_*h(f)=0$ and $h(f)\neq 0$. Choose $j$ to be the largest integer such that $\sigma_*^jh(f)\neq 0$, then $\sigma_*^{j+1}h(f)=\sigma_*(\sigma_*^jh(f))=0$. By Lemma \ref{kernelofsuspension}, $\sigma_*^jh(f)$ is a decomposable element. On the other hand, for $f^j:S^{n+j}\to B^jY$ being the adjoint of $f$, we have $h(f^j)=\sigma_*^jh(f)$ which means $\sigma_*^jh(f)$ is a spherical class, hence a primitive. Here, $B$ is the classifying space functor applied to the loop space $Y$ which satisfies $\Omega^jB^jY=Y$. Now, applying the above theorems to the decomposable primitive class $\sigma_*^jh(f)$ proves the claim that $\sigma_*^jh(f)=\zeta^2$ for some odd dimensional primitive class $\zeta$. Moreover, in the above choices for $Y$, the adojoint map $g:S^{n-1}\to \Omega Y$ satisfies 
$$h(g)=\sum Q^I\Sigma^{-1}\zeta_\alpha+D$$
where $D$ is a some of decomposable classes, and $\Sigma^{-1}\zeta_\alpha$ is a formal notations for the generators of $H_*\Omega Y$ where for $Y=Q\Sigma^2X$ or $Y=\Omega^lS^{n+l}$, $\Omega Y=Q\Sigma X$ or $\Omega Y=\Omega^{l+1}S^{n+l}$ justifies our notation. The equation $h(f)=\sigma_*h(g)$ yields the claimed expression for $h(f)$ as well as $h(f^j)$.
\end{proof}


\subsection{Snaith splitting and James-Hopf maps}
Recall that by James splitting \cite{James-reducedproduct} for a path connected space $X$ we have a splitting $\Sigma\Omega\Sigma X\simeq\bigvee_{r=1}^{+\infty}\Sigma X^{\wedge r}$. By projection on the $r$-th summand and taking adjoint we have a map $H_r:\Omega\Sigma X\to \Omega\Sigma X^r$ which we call James-Hopf invariant; we write $H$ to denote $H_2$. Moreover, this map can be composed with the stabilisation map to give a map $\Omega \Sigma X\to Q X^{\wedge r}$. This generalises as follows to yield stable maps. Recall that, for $k\geqslant 1$, when $X$ is path connected, we have Snaith splitting \cite{Snaith}
$$\Sigma^\infty \Omega^k\Sigma^kX\simeq\bigvee_{r=1}^{+\infty}\Sigma^\infty D_r(X,k)$$
{where $D_r(X,k)=F(\R^k,r)\ltimes_{\Sigma_r}X^{\wedge r}$, writing $D_rX=D_r(X,+\infty)$, where $F(\R^k,r)$ is the configuration space of $r$ distinct points in $\R^k$. Here, $F(\R^k,r)$ does not have a base point, hence we have a the description with half-smash; if we are to add a disjoint basepoint then the space $D_r(X,k)$ also can be identified with $F(\R^k,r)_+\wedge_{\Sigma_r}X^{\wedge r}$.} In particular, $D_2(S^n,k)=\Sigma^nP_{n}^{n+k-1}$, $D_2S^n\simeq\Sigma^n P_n$ \cite[Corollary 1.4]{Kuhngeometry}. Here, $P^n$ denotes the $n$-dimensional projective space, $P$ is the infinite dimensional real projective space, and $P_n=P/P^{n-1}$. By projection on to the $r$-th summand, and taking stable adjoint, we have a map
$$j_r^k:\Omega^k\Sigma^kX\to QD_r(X,k)$$
which we refer to as the $r$-th James-Hopf map for $\Omega^k\Sigma^kX$, and for $k=+\infty$ we write $j_r:QX\to D_rX$ for this map and call it stable James-Hopf maps. The maps $j_{r}^{k+d}$ and $\Omega^kj_r^k$ are compatible by through the suspension $\Omega^k\Sigma^k X\to \Omega^{k+d}\Sigma^{k+d}X$ where $d>0$ fitting into some obvious commutative diagrams \cite[Proposition 1.1, Theorem 1.2]{Kuhngeometry} (see also \cite{Milgram-unstable}).\\

\textbf{Desuspending $D_r(S^n,k)$.} For technical reasons which are essential to one of our main tools, namely Theorem \ref{detect1}, we are interested in the cases where $D_r(X,k)$ admits a desuspension, particularly when $X=S^{n}$ with $n>0$. There are various ways to see this. For instance, it is known that for $0<r<+\infty$, there is a $\Sigma_r$-equivariant diffeomorphism $F(\R^k,r)\to\R^k\times F(\R^k-\{0\},r-1)$ where $\Sigma_r$ acts trivially on $\R^k$
\cite[Lemma 5.7]{CMT-CX}. After one point compactification, the above homeomorphism induces a $\Sigma_r$-equivariant homeomorphism $F(\R^k,r)_+\to S^k\wedge F(\R^k-\{0\},r-1)_+$ where $\Sigma_r$ acts trivially on the $S^k$ factor. This allows to see that $D_r(X,k)$ admits at least one desuspension such that if $X$ is path connected CW-complex, then $D_r(X,k)=\Sigma(\Sigma^{-1}D_r(X,k))$ for some path connected CW-complex $\Sigma^{-1}D_r(X,k)$. In a similar vein, notice that there is a $\Sigma_r$-equivariant homeomorphism $(S^n)^{\wedge r}\to S^n\wedge (S^{r-1})^{\wedge n}$ where on the left $\Sigma_r$ acts by permutation and one the right it acts trivially on the $S^n$ factor while acting through reduced regular representation on the $S^{r-1}$ and diagonally on $(S^{r-1})^{\wedge n}$ factor. Consequently, $D_r(S^n,k)$ is homeomorphic to $\Sigma^n(F(\R^k,r)_+\wedge_{\Sigma_r}(S^{r-1})^{\wedge n})$ which is a path connected space either if $r>0$ or $n>2$ (in our applications $r>0$ is the case). The above homeomorphism also shows that if $X=\Sigma Y$, then $D_r(X,k)$ is a suspension. We also note that the space $X=S^n$, $D_r(X,k)=D_r(S^n,k)=F(\R^k,r)_+\wedge_{\Sigma_r}(S^n)^{\wedge r}$ is the Thom complex of the $n$-fold Whitney sum of the bundle
$$\xi_{k,r}:F(\R^k,r)\wedge_{\Sigma_r}\R^r\to F(\R^k,r)$$
Using this fact, it is possible to determine the precise number of desuspensions that $D_r(X,k)$ admits \cite{CCKN}. However, for the purpose of the present paper, having one suspension is enough.
For the purpose of future reference, we record this as a lemma.

\begin{lmm}\label{desuspend}
If $r>0$ or $n>2$, there exists a homeomorphism $D_r(S^n,k)$ to suspension of a path connected space $\Sigma^{-1}D_r(S^n,k)$.
\end{lmm}

The above desuspension, allows us to write $\Sigma^{-1}\xi$ for an element of $H_*\Sigma^{-1}D_r(X,k)$ which maps isomorphically to $\xi\in H_*D_r(S^n,k)$, bearing in mind that $\Sigma^{-1}D_r(S^n,k)$ is a space.\\

The homology of James-Hopf maps is well understood \cite{Kuhnhomology}. We are interested in the case of $X=S^n$ for $n>0$. {For $X$ a path connected space, t}here is a filtration on $H_*\Omega^k\Sigma^kX$, called the height filtration $\mathrm{ht}:H_*\Omega^k\Sigma^kX\to\mathbb{N}$, determined by
$$\mathrm{ht}(\xi\eta)=\mathrm{ht}(\xi)+\mathrm{ht}(\eta),\ \mathrm{ht}(Q^i\xi)=\mathrm{ht}(Q_j\xi)=2\mathrm{ht}(\xi).$$
The stable projection $\Sigma^\infty\Omega^k\Sigma^kX\to \Sigma^\infty D_r(X,r)$ induces projection on the elements of height $r$ in homology. {In the space case of $X=S^n$, the set of elements of $H_*\Omega^kS^{n+k}$ that are of height $r$ forms an additive basis for $\widetilde{H}_*D_r(S^n,k)$}. The homology of $j_r^k$ is compatible with this (see \cite{CLM} for example or \cite{BE3} for this), and in particular is determined by killing element of height lower than $r$ and mapping elements of height $r$ identically.


\subsection{Proof of Theorem \ref{detect1}}
Recall that by Boardman and Steer detecting a map $S^{2n+1}\to S^{n+1}$ by $Sq^{n+1}$ in its mapping cone, corresponds to detecting the adjoint mapping $S^{2n}\to\Omega S^{n+1}$ by $H:\Omega S^{n+1}\to \Omega S^{2n+1}$ \cite{BoardmanSteer} and the latter is equivalent to detecting the adjoint map $\widetilde{f}:S^{2n}\to\Omega S^{n+1}$ in homology by $h(\widetilde{f})=x_n^2$ \cite[Proposition 6.1.5]{Harper}. A similar and more general statement holds on maps $f:S^{2n}\to QX$ saying that if $h(f)=x_n^2$ with $x_n\in\widetilde{H}_nX$ then the stable adjoint of $f$, $S^n\to X$ is detected by $Sq^{n+1}x_n=x_{2n+1}$ in its stable mapping cone \cite[Proposition 5.8]{AsadiEccles}.

Now, we are able to prove Theorem \ref{detect1}. We shall use the notation fixed above, provided by the desuspension of Lemma \ref{desuspend}.

\begin{thm}
Suppose $f:S^{2d}\to\Omega^lS^{n+l}$ is given with $h(f)=(\sum Q_Jx_n)^2$. Then, the composition
$$(\Omega j_{2^{l_0}}^{l-1})\circ f:S^{2d}\to\Omega^lS^{n+l}\to\Omega QD_{2^{l_0}}(S^{n+1},l-1)$$
satisfies
$$h(\Omega j_{2^{l_0}}^{l-1}f)=(\sum_{l(J)=l_0} \Sigma^{-1}\overline{Q_Jx_{n+1}})^2$$
and its stable adjoint, say $g:S^{2d+1}\not\to D_{2^{l_0}}(S^{n+1},l-1)$, is detected by $Sq^{d+1}$ on a $(d+1)$-dimensional class in its stable mapping cone. Here, $l_0=\max\{l(J)\}$.
\end{thm}

\begin{proof}
Suppose $f:S^{2d}\to\Omega^{l}S^{n+l}$ is given with $h(f)=(\sum Q_Jx_n)^2$. Let $l_0=\max\{l(J)\}$ where the maximum is over all sequences $I$ in the expression for $h(f)$. By \cite[Theorem 2.4]{Zare-Els-1} all sequences $J=(j_1,\ldots,j_s)$, $Q_Jx_n$ must be odd dimensional, and in the above expression should be strictly increasing with $j_k+j_{k+1}$ being odd for all $k\in\{1,\ldots,s-1\}$ whereas in an $l$-fold loop spaces we have $j_k<l$. Moreover, $j_1\geqslant 1$ for all sequence $J$ in the sequence. Consider James-Hopf map $j_{2^{l_0}}^{l-1}:\Omega^{l-1}S^{n+l}\to QD_{2^{l_0}}(S^{n+1},l-1)$. In homology, this kills all $Q_Jx_{n+1}$ with $l(J)<l_0$ whereas it acts as projection when applied to classes of the form $Q_Jx_{n+1}$ with $l(J)=l_0$. We then have,
$$(\Omega j_{2^{l_0}}^{l-1})_*\sigma_*(\sum Q_Jx_n)=(\Omega j_{2^{l_0}}^{l-1})_*(\sum Q_{J-1}x_{n+1})=\sum_{l(J)=l_0}\overline{Q_{J-1}x_{n+1}}$$
where $J-1=(j_1-1,\ldots,j_s-1)$. For the composition
$$\Omega j_{2^{l_0}}^{l-1}f:S^{2d}\to \Omega^lS^{n+l}\to \Omega QD_{2^{l_0}}(S^{n+1},l-1)=Q\Sigma^{-1}D_{2^{l_0}}(S^{n+1},l-1)$$
we have
$$h(\Omega j_{2^{l_0}}^{l-1}f)=(\sum_{l(J)=l_0} \Sigma^{-1}\overline{Q_Jx_{n+1}})^2$$
with $\overline{Q_Jx_{n+1}}$ being the image of $Q_Jx_{n+1}$ under the projection $\Sigma^\infty\Omega^{l-1}S^{n+l}\to\Sigma^\infty D_{2^{l_0}}(S^{n+l},l-1)$ provided by Snaith splitting. Consequently, by \cite[Proposition 5.8]{AsadiEccles} quoted above, the stable adjoint of $\Omega j_{2^{l_0}}^{l-1}f$, say $g:S^{2d+1}\not\to D_{2^{l_0}}(S^{n+l},l-1)$ is detected by $Sq^{d+1}$ on a $d+1$ dimensional class in its stable mapping cone.
\end{proof}

\section{Towards a generalised Browder theorem: Proof of Theorem \ref{degenerate}}\label{Browder}
This section is dedicated to a proof of Theorem \ref{degenerate}. This provides a generalisation of Browder theorem, when stated in terms of existence of spherical classes. We wish to eliminate the classes $(\sum Q_Jx_n)^2\in H_*\Omega^lS^{n+l}$ corresponding to `degenerate' cases from being spherical, unless cases coming from inclusion of the bottom cell of Hopf invariant one elements. Note that by \cite[Lemma 3.6]{Zare-Els-1} and \cite[Theorem 2.4]{Zare-Els-1}, if $h(f)=\xi^2\in H_{2d}\Omega^lS^{n+l}$ then $d$ is odd and $h(f)=(\sum Q_Jx_n)^2\in H_*\Omega^lS^{n+l}$ for certain strictly increasing sequences $J$. Broadly speaking, $\xi^2$ is counted as a degenerate case if $\dim\xi\neq 2^t-1$.  The following observation, eliminates some of these cases under some conditions on the sequences $J$ involved in the expression.

\begin{thm}\label{degenerate1}
Suppose $d$ is an odd positive integer with $d\neq 2^t-1$ for all $t>0$, $d+1\equiv2\mathrm{mod}4$, $l>3$. Suppose $f:S^{2d}\to\Omega^{l}S^{n+l}$ is given with $h(f)=(\sum Q_Jx_n)^2$. Then $l>n$. Moreover, in the case $l>n$, we have $l(J)\leqslant l-n-\frac{l}{2^{l_0}}$ for all $J$ involved in $h(f)$ where $l_0=\max\{l(J)\}$.
\end{thm}

\begin{proof}
Suppose $n\geqslant l$. We show this leads to a contradiction.

By Theorem \ref{detect1}, the equation $h(f)=(\sum Q_Jx_n)^2$ implies that the stable adjoint of $\Omega j_{2^{l_0}}^{l-1}\circ f$, say $g:S^{2d+1}\not\to D_{2^{l_0}}(S^{n+1},l-1)$ is detected by $Sq^{d+1}$ on a $(d+1)$ dimensional class in its stable mapping cone $D_{2^{l_0}}(\R^{l-1},S^{n+1})\cup_g e^{2d+2}$. The dimension of the top class in $H_*D_{2^{l_0}}(\R^{l-1},S^{n+1})$ as well as the top dimensional cell of this complex is equal to
$$\mathrm{top}:=\dim \overbrace{Q_{l-2}\cdots Q_{l-2}}^{l_0-\textrm{times}}x_{n+1}=(2^{l_0}-1)(l-2)+2^{l_0}(n+1)=(2^{l_0}-1)(l-1)+2^{l_0}n+1.$$
On the other hand, the strictly increasing sequences $J$ with $l(J)=l_0$ for which $Q_Jx_{n}$ would be of minimum dimensions, respectively, is
$J=(1,\ldots,l_0)$
with
$$d_0:=\min d=\dim Q_1\cdots Q_{l_0}x_{n+1} =2^{{l_0}-1}(l_0-1)+1+2^{l_0}n$$
Now, we may compute that
$$2d_0+1-\mathrm{top}=2^{l_0}(l_0+n-l)+l+1.$$
The inequality $2d+1-\mathrm{top}>2$ if and only if $2^{l_0}(l_0+n-l)+l+1>2$ if and only if
$$2^{l_0}(l_0+n-l)+l\geqslant 2.$$
As $l>3$, for $n\geqslant l$, the above inequality obviously holds which implies that $H^{2d}(D_{2^{l_0}}(S^{n+1},l-1)\simeq 0$, and $H^{2d-1}(D_{2^{l_0}}(S^{n+1},l-1)\simeq 0$. By Lemma \ref{main1}, it is impossible to have such an $f$ with $h(f)$ as claimed above. This is a contradiction. Hence, it is impossible to have $n\geqslant l$, and we have $l>n$.\\
For $n<l$,  the inequality $2^{l_0}(l_0+n-l)+l\geqslant 2$ yields
$$l_0\geqslant l-n-\frac{l+2}{2^{l_0}}$$
which if holds agains makes the existence of $f$ impossible. Hence, in the case of $n<l$, we must have
$$l_0\leqslant l-n-\frac{l}{2^{l_0}}<l-n$$
which is the claimed necessary condition. But, by definition of $l_0$, this means that for all $J$ in the above expression
$$l(J)\leqslant l-n-\frac{l}{2^{l_0}}<l-n.$$
This completes the proof.
\end{proof}

The second part of Theorem \ref{degenerate} is a consequence of the above theorem together with Freudenthal's suspension theorem. We have the following.
\begin{crl}\label{degenerate2}
If $f\in\pi_{2d}QS^n$ is given with $h(f)\neq 0$ and $\sigma_*h(f)=0$ then $2d\geqslant 3n$. Moreover, there exists $l>n$ so that $f$ admits a pull back to $\pi_d\Omega^lS^{n+l}$.
\end{crl}

\begin{proof}
We have
$$h(f)=(\sum Q_Jx_n)^2$$
where the sequences $J$ satisfy in the conditions of \cite[Theorem 2.4]{Zare-Els-1}. By Freudenthal's theorem, $f$ pulls back to an element  $f'\in\pi_{2d}\Omega^{2d-2n+1}S^{(2d-2n+1)+n}$. Since the stabilisation $\Omega^{2d-2n+1}S^{(2d-2n+1)+n}\to QS^n$ induces a monomorphism, we deduce that
$$h(f')=(\sum Q_Jx_n)^2\neq 0.$$
By Theorem \ref{degenerate1}, to have the above equality we need $2d-2n+1>n$ which implies that $2d\geqslant 3n$.
\end{proof}

\subsection{Partial results on degenerate cases with $l>n$}
Throughout the subsection we assume that $d+1\neq 2^t$ for all $t>0$, $d$ an odd number, and $l>n$. It is possible to eliminate some small values of $l_0$ in the case of $l>n$. First, note that according to \cite[Theorem 2.4]{Zare-Els-1} if $f:S^{2d}\to \Omega^lS^{n+l}$ is given with $\sigma_*h(f)=0$ then $d$ is odd and
$$h(f)=(\sum Q_Jx_n)^2$$
where $J$ is strictly increasing with $j_1$ being odd as well as $j_k+j_{k+1}$ for all $k$. This is equivalent to $h(f)=(\sum Q^{I_J}x_n)^2$ with $I_J$ only consisting of odd entries. Here, $Q^{I_J}x_n=Q_Jx_n$. Moreover, recall that for once and all we have eliminated the case of $l(J)=0$ corresponding to the empty sequence in Lemma \label{l=0}. We have the following.
\begin{lmm}\label{l>n.eliminate1}
Suppose $f:S^{2d}\to \Omega^lS^{n+l}$ is given $h(f)=(\sum Q_Jx_n)^2$ where $l(J)>1$ for all $J$.\\
(i) If $n$ is even then $l(J)\geqslant 1$, and $l(J)$ is odd for all $J$.\\
(ii) If $n$ is odd then $l(J)\geqslant 2$, and $l(J)$ is even for all $J$.
\end{lmm}

\begin{proof}
The class $Q^{I_J}x_n=Q_Jx_n$ must be odd dimensional with $I$ having only odd entries. The lemma now follows noting that for $I=(i_1,\ldots,i_s)$, $\dim Q^{I_J}x_n=\sum_{k=1}^s i_k+n$.
\end{proof}

\begin{lmm}\label{l>n.eliminate2}
Let $n$ be even, and $d+1\equiv 2\textrm{ mod }4$. Suppose $f:S^{2d}\to \Omega^lS^{n+l}$ is given with $h(f)=(\sum Q_Jx_n)^2$ where $l(J)>1$ for all $J$. Then $l_0\geqslant 3$.
\end{lmm}

\begin{proof}
Suppose $l_0=1$ and $n$ is even. According to Theorem \ref{detect1}, the mapping $g:S^{2d+1}\to D_{2^{l_0}}(S^{n+1},l-1)$ is detected by $Sq^{d+1}$ on a $d+1$ dimensional class. As $d+1\neq 2^t$ then, for $d+1=4m+2$, we have $Sq^{d+1}=Sq^2Sq^{4m}+Sq^1Sq^{4m}Sq^1$. Since $D_{2^{l_0}}(S^{n+1},l-1)=\Sigma^{n+1}P_{n+1}^{n+l-1}$ then a $d+1$ dimensional class in $H^*D_{2^{l_0}}(S^{n+1},l-1)$ corresponds to $(d-n)$-dimensional class $a^{d-n}\in H^*P_{n+1}^{n+l-1}$. We then have
$$Sq^{d+1}=Sq^2Sq^{4m}\Sigma^{n+1}a^{d-n}+Sq^1Sq^{4m}Sq^1\Sigma^{n+1}a^{d-n}.$$
We have $H^*C_g\simeq H^*\Sigma^{n+1}P_{n+1}^{n+l-1}$ in dimensions $\leqslant 2d+1$. Now, by stability of the Steenrod operations and for dimensional reasons, noting that $4m=d-1$ and $4m+1=d$ we have
$$Sq^{4m}\Sigma^{n+1}a^{d-n}=\Sigma^{n+1}Sq^{4m}a^{d-n}=0,\ Sq^{4m}Sq^1\Sigma^{n+1}a^{d-n}=\Sigma^{n+1}Sq^{4m}Sq^1a^{d-n}=0$$
in $H^*\Sigma^{n+1}P_{n+1}^{n+l-1}$. Consequently, $Sq^{d+1}\Sigma^{n+1}a^{d-n}=0$ which is a contradiction to Theorem \ref{detect1}. On the other hand, note that for $Q^Ix_n=Q_Jx_n$ being of odd dimension with $n$ even and $I$ having only odd entries, we need $l(I)=l(J)$ to be odd and $l_0=2$ gives an even dimensional class. Consequently, $l_0\geqslant 3$.
\end{proof}

\section{Proof of Theorem \ref{main2}}
Suppose $f:S^d\to\Omega^lS^{n+l}$ is given with $h(f)\neq 0$ where $3<l<9$ and $n>0$. We wish to show that it is either corresponding to the inclusion of the bottom cell $S^n\to\Omega^lS^{n+l}$, or ``is'' a Hopf invariant one element. We consider to separate cases: (1) $\sigma_*h(f)\neq 0$ the case which follows from an inductive argument, assuming the Theorem for the $(l-1)$-fold loop spaces; (2) $\sigma_*h(f)=0$ which by \cite[Theorem 2.4]{Zare-Els-1} we know that $$h(f)=(\sum\epsilon_J Q_Jx_n)^2$$
where $Q_Jx_n$ is odd dimensional, and $J=(j_1,\ldots,j_s)$ (including the empty sequence) is strictly increasing with $j_1$ and $j_k+j_{k+1}$ odd with $k\in\{1,\ldots,s-1\}$. By an inductive argument, it is enough to eliminate the cases $f:S^{2d}\to\Omega^lS^{n+l}$ with $d$ odd from giving rise to a square spherical class. Moreover, recall that we have completely determined spherical classes in $H_*\Omega^lS^{l+n}$ for $l<4$
(see \cite[Theorem 1.11]{Zare-PEMS} and \cite[Theorem 2.2]{Zare-Els-1}. Using \cite[Theorem 2.4]{Zare-Els-1} we provided a table for all potential spherical classes in $H_*\Omega^lS^{l+n}$ with $4\leqslant l\leqslant 9$. Recall that
\begin{lmm}\label{Lemma-table}
(\cite[Lemma 8.1]{Zare-Els-1}) Suppose $f\in\pi_i\Omega^lS^{n+l}$ with $l<9$ so that $\sigma_*h(f)=0$. Then,
$$h(f)=(Q_Jx_n)^2$$
where $J$ can be chosen from the table below. In particular, $Q_Jx_n$ has to be $A$-annihilated.
\begin{center}
\begin{tabular}{|c|l|c|l|c|c|c|c}
\hline
$J$                                   & $Q^Ix_n$                   & dimension & \textrm{Trivial eliminations}\\
\hline
$\phi=()$                             & $x_n$                      & $n$    & \\
\hline
$(j)$                                 & $Q^{j+n}x_n$               & $j+2n$    & \textrm{eliminated if $dim=j+2n$ is}\\
\textrm{ with }$0<j\leqslant7$       &                            &           & \textrm{even or by $Sq^1_*$ if $j+n$ is even}\\
\hline
$(1,j)$                              & $Q^{1+j+2n}Q^{j+n}x_n$     & $1+2j+4n$ & \textrm{eliminated by $Sq^2_*$ if $n$ is even}\\
with $j=2,4,6$  &&&\\
\hline
$(3,j)$                              & $Q^{3+j+2n}Q^{j+n}x_n$            & $3+2j+4n$ & \textrm{eliminated by $Sq^2_*$ if $n$ is even}\\
with $j=4,6$ &&&\\
\hline
$(5,6)$                              & $Q^{11+2n}Q^{6+n}x_n$            & $17+4n$ & \textrm{eliminated by $Sq^2_*$ if $n$ is even}\\
\hline
$(1,2,j)$                            & $Q^{3+2j+4n}Q^{2+j+2n}Q^{j+n}x_n$ & $5+4j+8n$ & \textrm{eliminated by $Sq^4_*$ if $n$ is odd}\\
with $j=3,5,7$ &&&\\
\hline
$(1,4,j)$                            & $Q^{5+2j+4n}Q^{4+j+2n}Q^{j+n}x_n$ & $9+4j+8n$ & \textrm{eliminated by $Sq^4_*$ if $n$ is odd}\\
with $j=5,7$ &&&\\
\hline
$(1,6,7)$                & $Q^{21+4n}Q^{13+2n}Q^{7+n}x_n$                & $41+8n$   & \textrm{eliminated by $Sq^4_*$ if $n$ is odd}\\
\hline
$(3,4,j)$                & $Q^{7+2j+4n}Q^{4+j+2n}Q^{j+n}x_n$             & $11+4j+8n$& \textrm{eliminated by $Sq^4_*$ if $n$ is odd}\\
with $j=5,7$ &&&\\
\hline
$(3,6,7)$                & $Q^{23+4n}Q^{13+2n}Q^{7+n}x_n$                & $43+8n$   & \textrm{eliminated by $Sq^4_*$ if $n$ is odd}\\
\hline
$(5,6,7)$                & $Q^{25+4n}Q^{13+2n}Q^{7+n}x_n$                & $45+8n$   & \textrm{eliminated by $Sq^4_*$ if $n$ is odd}\\
\hline
$(1,2,3,4)$              & $Q^{25+8n}Q^{13+4n}Q^{7+2n}Q^{4+n}x_n$        & $49+16n$ & \textrm{eliminated by $Sq^{2^3}_*$ if $n$ is even}\\
\hline
$(1,2,3,6)$              & $Q^{33+8n}Q^{17+4n}Q^{9+2n}Q^{6+n}x_n$        & $65+16n$ & \textrm{eliminated by $Sq^{2^3}_*$ if $n$ is even}\\
\hline
$(1,2,5,6)$              & $Q^{37+8n}Q^{19+4n}Q^{11+2n}Q^{6+n}x_n$       & $73+16n$ & \textrm{eliminated by $Sq^{2^3}_*$ if $n$ is even}\\
\hline
$(3,4,5,6)$              & $Q^{41+8n}Q^{21+4n}Q^{11+2n}Q^{6+n}x_n$       & $79+16n$ & \textrm{eliminated by $Sq^{2^3}_*$ if $n$ is even}\\
\hline
$(1,2,3,4,5)$            & $Q^{65+16n}Q^{33+8n}Q^{17+4n}Q^{9+2n}Q^{5+n}x_n$ & $129+32n$ & \textrm{eliminated by $Sq^{2^4}_*$ if $n$ is odd}\\
\hline
$(1,2,3,4,7)$            & $Q^{81+16n}Q^{41+8n}Q^{21+4n}Q^{11+2n}Q^{7+n}x_n$ & $161+32n$ & \textrm{eliminated by $Sq^{2^4}_*$ if $n$ is odd}\\
\hline
$(3,4,5,6,7)$            & $Q^{97+16n}Q^{49+8n}Q^{25+4n}Q^{13+2n}Q^{7+n}x_n$  & $191+32n$& \textrm{eliminated by $Sq^{2^4}_*$ if $n$ is odd}\\
\hline
$(1,2,3,4,5,6)$          & $Q^{161+32n}Q^{81+16n}Q^{41+8n}$              & $301+64n$ & \textrm{eliminated by $Sq^{2^5}_*$ if $n$ is even}\\
                         & $Q^{21+4n}Q^{11+2n}Q^{6+n}x_n$                &           & \\
\hline
$(1,2,3,4,5,6,7)$        & $Q^{385+64n}Q^{193+32n}Q^{97+16n}Q^{49+8n}$  & $769+128n$& \textrm{eliminated by $Sq^{2^6}_*$ if $n$ is odd}\\
                         & $Q^{25+4n}Q^{13+2n}Q^{7+n}x_n$                &           & \\
\hline
\end{tabular}
\end{center}
\end{lmm}
We note that the case $J=()$ was not listed in \cite[Lemma 8.1]{Zare-Els-1} as it was considered rather trivial. As a consequence, in order to verify Eccles conjecture on $H_*\Omega^lS^{n+l}$ for $l<9$ and $n>0$, it is enough to eliminate the classes listed in the above table. For the cases with $3<l<9$, Lemma \ref{Lemma-table} tabulates all possible cases. In this section, we eliminate these cases from giving rise to spherical classes, by appealing to Lemma \ref{main1} and Theorem \ref{detect1}. Finally, we note that previously in \cite{Zare-Els-1}, we have computed spherical classes in $H_*\Omega^lS^{n+l}$ for $n>0$ and $l=1,2,3$ of which the case $l=3$ is related to the base of our induction in this paper.

\subsection{The case of the empty sequence}
We being with the easiest case of Theorem \ref{main1} corresponding to the empty sequence, i.e. $J=()$.

\begin{lmm}\label{l=0}
Let $f:S^{2n}\to \Omega^lS^{n+l}$ with $n>0$ and $l>0$ be given with $h(f)=x_n^2$. Then, $f$ is image of a Hopf invariant one element under the stabilisation $\Omega S^{n+1}\to \Omega^lS^{n+l}$.
\end{lmm}

\begin{proof}
If $x_n^2$ is spherical in $\Omega^lS^{n+l}$, the for $h(f)=x_n^2$ we have $f\in\pi_{2n}\Omega^lS^{n+l}\simeq\pi_{2n+l}S^{n+l}$. By Freudenthal theorem, the mapping $f:S^{2n}\to\Omega^lS^{n+l}$ pulls back to $\widetilde{f}:S^{2n}\to \Omega S^{n+1}$. As $h(f)\neq 0$, then $h(\widetilde{f})\neq 0$ which shows that $h(\widetilde{f})=x_n^2$. This is the same as the adjoint of $\widetilde{f}$, $S^{2n+1}\to S^{n+1}$ being detected by the unstable Hopf invariant. Consequently, $n\in\{1,3,7\}$. In this case, $f$ has to be the image of a Hopf invariant one element under the stabilisation map $\Omega S^{n+1}\to\Omega^lS^{n+l}$. This completes the proof.
\end{proof}

\subsection{The case of $l(J)=1$}
\begin{lmm}\label{l=1}
For $l<9$, the classes $(Q_jx_n)^2=(Q^{j+n}x_n)^2$ with $0<j<8$ are not spherical in $H_*\Omega^lS^{n+l}$.
\end{lmm}

\begin{proof}
Since the stabilisation $\Omega^lS^{n+l}\to\Omega^8S^{n+8}$ induces a monomorphism in homology, then it is enough to eliminate the aforementioned classes from being spherical in $H_*\Omega^8S^{n+8}$.\\
By \cite[Lemma 3.6]{Zare-Els-1}, we are interested in the cases with $Q_jx_n$ being odd dimensional. As $\dim(Q_jx_n)=2n+j$, this means that $j$ has to be odd. Moreover, if $j+n$ is even, then by the Nishida relations $Sq^1_*Q^{2t}=Q^{2t-1}$, we see that $Sq^1_*Q^{j+n}x_n=Q^{j+n-1}x_n\neq 0$. The Cartan formula,$Sq^{2t}_*\xi^2=(Sq^t_*\xi)^2$ then shows that $Sq^2_*(Q^{j+n}x_n)^2\neq 0$. Hence, it is not possible to have $j+n$ even, which with $j$ being odd, implies that $n$ must be even.\\
Now, consider James-Hopf map $j_2^{l-j} :\Omega^{l-j}\Sigma^{l-j}S^{n+j}\to Q\Sigma^{n+j}P_{n+j}^{n+7}$ which sends $x_{n+j}^2$ to $\Sigma^{n+j}a_{n+j}$ in homology, where $x_i$ and $a_i$ denotes generators of $\widetilde{H}_iS^n$ and $\widetilde{H}_iP_k^{k+n}$ respectively. For the iterated homology suspension $\sigma_*^j:H_*\Omega^lS^{n+l}\to H_{*-j}\Omega^{l-j}S^{n+l}$ we have $\sigma_*^jQ^{n+j}x_n=x_{n+j}^2$. A simple computation together with dimensional reasons shows that
$$(\Omega^jj_2^{l-j} )_*(Q^{n+j}x_n)=\Sigma^na_{n+j}.$$
If $f:S^{4n+2j}\to \Omega^lS^{n+l}$ is any mapping with $h(f)=(Q^{n+j}x_n)^2$, then as $\Omega^jj_2^{l-j} $ is a loop map, we have
$$h(\Omega^jj_2^{l-j} \circ f)=(\Omega^jj_2^{l-j} )_*(Q^{n+j}x_n)^2=(\Sigma^na_{n+j})^2\in H_*Q\Sigma^nP_{n+j}^{n+7}.$$
Now, we divide the proof depending on the odd values for $j$ between $1$ and $8$.\\

\textbf{Case }$j=7$. In this case, $P_{n+j}^{n+7}=S^{n+7}$, and we have $h(\Omega^jj_2^{l-j} \circ f)=x_{2n+7}^2$. This implies that $g:=\Omega^jj_2^{l-j} \circ f$ is a Hopf invariant one element, detected by $Sq^{2n+8}$, which also requires $2n+8=2,4,8$. But, this implies that $n<1$ which is not the case that is considered in this paper. So, we are done.\\

\textbf{Case }$j=1,5$. In this case, since $n$ is even then $2n+j+1=4m+2$ for some $m$. By \cite[Proposition 5.8]{AsadiEccles} the above equality implies that the stable adjoint of $\Omega^jj_2^{l-j} \circ f$, say $g:S^{4n+2j+1}\to \Sigma^{n+1}P_{n+j}^{n+7}$ is detected by $Sq^{2n+j+1}$ on a $2n+j+1$ dimensional class in its mapping cone, say $\Sigma^{n+1}a^{n+j}\in H^{2n+j+1}\Sigma^{n+1}P_{n+j}^{n+7}$. For $2n+j+1=4m+2$, we have $Sq^{4m+2}=Sq^2Sq^{4m}+Sq^1Sq^{4m}Sq^1$. However, for dimensional reasons, we have
$$Sq^{4m}\Sigma^{n+1}a^{n+j}=\Sigma^{n+1}Sq^{4m}a^{n+j}=0,\ Sq^{4m}Sq^1\Sigma^{n+1}a^{n+j}=Sq^{n+1}Sq^{4m}Sq^1a^{n+j}=0$$
in $H^*\Sigma^{n+1}P_{n+j}^{n+7}\simeq H^*C_g$ for $*<4n+2j+2$. Consequently, $Sq^{4m+2}\Sigma^{n+1}a^{n+j}=0$ in $H^*C_g$ which is a contradiction. This proves the lemma in this case. \\

\textbf{Case }$j=3$. In this case, we consider situations where $2n+j+1=2^{t+1}$ and otherwise.\\
\textbf{Case }$j=3$ and $2n+j+1=2^{t+1}$. For $j=3$, the equation $2n+j+1=2^t$ yields $n=2^t-2$. In this case, we have
$$h(f)=(Q^{2^t+1}x_{2^t-2})^2.$$
However, by Nishida relations we may compute that $Sq^2_*Q^{2^t+1}x_{2^t-2}=Q^{2^t-1}x_{2^t-2}\neq 0$ in $H_*\Omega^lS^{n+l}$ which together with Cartan formula shows that
$$Sq^4_*(Q^{2^t+1}x_{2^t-2})^2=(Sq^2_*Q^{2^t+1}x_{2^t-2})^2=(Q^{2^t-1}x_{2^t-2})^2\neq 0.$$
This contradicts the fact that $h(f)$ has to be $A$-annihilated, and completes the proof.\\
\textbf{Case }$j=3$ and $2n+j+1\neq 2^t$, $n\equiv 2\mathrm{mod}4$. In this case, $n+3\equiv1\mathrm{mod}4$. Applying Nishida relations we compute that $Sq^2_*Q^{n+3}x_n=Q^{n+1}x_n$ which consequently, together with Cartan formula, shows that $Sq^4_*(Q^{n+3}x_n)^2=(Sq^2_*Q^{n+3}x_n)^2=(Q^{n+1}x_n)^2\neq 0$ which contradicts that $h(f)$ has to be $A$-annihilated. This eliminates this case.\\
\textbf{Case }$j=3$, $2n+j+1\neq 2^t$ and $n\equiv0\mathrm{mod}4$. If $(Q^{n+3}x_n)^2\in H_*\Omega^8S^{n+8}$ is spherical then there exists $f:S^{4n+6}\to\Omega^8S^{n+8}$ with $h(f)=(Q^{n+j}x_n)^2$. In this case, the stable adjoint of $\Omega^jj_2^{l-j} \circ f:S^{4n+6}\to Q\Sigma^{n}P_{n+3}^{n+7}$, say
$g:S^{3n+6}\to P_{n+3}^{n+7}$ is detected by $Sq^{2n+4}$ on a $(n+3)$-dimensional class in its stable mapping cone $C_g=P_{n+3}^{n+7}\cup_{g}e^{3n+7}$, say $a^{n+3}\in H^{n+3}C_g\simeq H^{n+3}P_{n+3}^{n+7}$. Since $2n+4\neq 2^t$ for all $t$, then $Sq^{2n+4}$ can be written as a linear combination of terms of the form $Sq^{2^{t_1}}\cdots Sq^{2^{t_s}}$, i.e.
$$Sq^{2n+4}=\sum_{s>1}\epsilon_{t_1,\ldots,t_s}Sq^{2^{t_1}}\cdots Sq^{2^{t_s}}$$
where $t_i\geqslant 0$ and $\epsilon_{t_1,\ldots,t_s}\in\Z/2$. \\
First, note that for $s>4$ the equation $2n+4=\sum_{i=1}^s 2^{t_i}$ together with the simple observation $Sq^1Sq^1=0$ shows that if $t_i=0$ then $t_{i+1}>0$ as well as $t_{i_1}>0$. In this situation $n+7<\dim(Sq^{2^{t_{s-3}}}Sq^{2^{t_{s-2}}}Sq^{2^{t_{s-1}}}Sq^{t_s}a^{n+3})<3n+7-2^{t_{s-4}}$ where $C_g$ has no cells in these dimensions. This implies that $Sq^{2^{t_{s-3}}}Sq^{2^{t_{s-2}}}Sq^{2^{t_{s-1}}}Sq^{t_s}a^{n+3}=0$. Consequently, $Sq^{2n+4}a^{n+3}\neq 0$ implies that there is a terms with $s\leqslant 4$.\\
By inspection, similar as above, we see that the only possible term when $s=4$ is of the form $Sq^{2^t}Sq^1Sq^2Sq^1$. However, as $n\equiv0\mathrm{mod}4$ we see that $Sq^2Sq^1a^{n+3}=Sq^2a^{n+4}=0$. This eliminates this case.\\
Note that $s>1$ (equivalently $2n+4\neq 2^t$) together with $t_1\geqslant 0$ implies that if $t_s>2$ then $n+7<\dim(Sq^{2^{t_s}}a^{n+3})<3n+7$ which noting that $C_g$ has no cell in these dimensions implies that $Sq^{2^{t_s}}a^{n+3}=0$ in $H^*C_g$. If, $t_s>2$ for all terms in the above expression for $Sq^{2n+4}$ then $Sq^{2n+4}a^{n+3}=0$ in $H^*C_g$ which is a contradiction. Hence, in order for $Sq^{2n+4}a^{n+3}\neq 0$ in $H^*C_g$, there must exists a term with $t_s\in\{0,1,2\}$.\\
$t_s\neq 0$: Since $\sum 2^{t_i}=2n+4$, then $t_s=0$ would imply that there exists $1\leqslant i<s-1$ with $t_i=0$; if $t_{s-1}=0$ then we have $Sq^1Sq^1=0$. So, $t_{s-1}>0$. If $t_s>1$, then $Sq^{t_{s-1}}Sq^{t_s}a^{n+3}=0$ in $H^*C_g$ for dimensional reasons as above. So, there is a term with $t_{s-1}=1$. But, since $n\equiv 0\mathrm{mod}4$, then
$$Sq^{2}Sq^1a^{n+3}=Sq^2a^{n+4}={n+4\choose 2}a^{n+6}=0.$$
This eliminates the case of $t_s=0$. \\
Now, we are left with the cases $s=2,3$ and $t_s=1,2$. By inspection, we see that the only remaining cases are
$$Sq^{2^{t}}Sq^4,\ Sq^{2^{t}}Sq^1Sq^2.$$
However, the second terms is not a suitable choice as for $t_1>0$ this is not in an even grading of the Steenrod algebra, and for $t_1=0$ this implies that $3n+7=n+7$ which is not possible for $n>0$. Hence, we need to eliminate the case $Sq^{2^t}Sq^4$. If $n\equiv0\mathrm{mod}8$ then
$$Sq^{4}a^{n+3}=0$$
implying that $Sq^{2^t}Sq^4a^{n+3}=0$. This is a contradiction which leaves us with the case $n\equiv4\mathrm{mod}8$. But, as here $n=2^t$, this is possible only if $n=4$. However, this is a contradiction as it implies $Sq^8=Sq^4Sq^4$ where $8=2^3$ and it shouldn't be possible to decompose $Sq^{2^3}$ as a composition of operations of lower order. This eliminates this case, and completes the proof.
\end{proof}

\subsection{Degenerate cases for $l(J)>1$; $d+1\equiv 2\textrm{ mod }4$}
We begin by collecting some values in a table. We shall isolate the degenerate cases from non-degenerate cases. We wish to eliminate the classes in the following table from being spherical in $H_*\Omega^lS^{n+l}$ for $3<l<9$. Since the stabilisation $\Omega^lS^{n+l}\to\Omega^{l+1}S^{n+l+1}$ induces a monomorphism in homology, so in this paper, it is enough to eliminate the classes of the following table from being spherical in $H_*\Omega^8S^{n+8}$ which sets the default value $l=8$. Since, in each dimension there is a single sequence $J$, so $l_0=l(J)$. We also have set $\mathrm{top}=\dim D_{2^{l_0}}(S^{n+1},l-1)$.
\begin{center}
\begin{tabular}{|c|c|c|c|c|c|c|}
\hline
$J$               &  $d:=\dim Q^Ix_n$ & $\textrm{cases with } d+1=2^t$  & $d+1\equiv 2\textrm{ mod }4$ & $\mathrm{top}$  & $2d+1-\mathrm{top}$        \\
\hline
$(1,j)$           &  $1+2j+4n$        & none                            & \checkmark &  $22+4n$             & $4n+4j-19>2$               \\
with $j=2,4,6$    &                   &                                 &            &                      & if $j=6$ or \\
                  &                   &                                 &            &                      & $(j=4;n>1)$ or\\
                  &                   &                                 &            &                      & $(j=2;n>3)$\\
\hline
$(3,j)$           &  $3+2j+4n$        & $(j=6;n=2^{t-2}-4)$             & $\times$   & $22+4n$             & $4n+4j-15>2$            \\
with $j=4,6$      &                   &                                 &              &                     &\\
\hline
$(5,6)$           &  $17+4n$          & none                            & \checkmark & $22+4n$             & $4n+9>2$                  \\
\hline
$(1,2,j)$         &  $5+4j+8n$        & none                            & \checkmark & $50+8n$             & $8n+8j-39>2$              \\
with $j=3,5,7$    &                   &                                 &            &                     & if $j=5,7$ or \\
                  &                   &                                 &            &                     & $(j=3;n>2)$\\
\hline
$(1,4,j)$         &  $9+4j+8n$        & none                            & \checkmark & $50+8n$             & $8n+8j-31>2$              \\
with $j=5,7$      &                   &                                 &            &                     &\\
\hline
$(1,6,7)$         & $41+8n$           & none                            & \checkmark & $50+8n$             & $8n+32>2$               \\
\hline
$(3,4,j)$         & $11+4j+8n$        & $(j=5;n=2^{t-3}-4)$             & $\times$   & $50+8n$             & $8n+8j-27>2$  \\
with $j=5,7$      &                   & $(j=7;n=2^{t-3}-5)$             &            &                     &                          \\
\hline
$(3,6,7)$         & $43+8n$           & none                            & $\times$   & $50+8n$             & $8n+37>2$  \\
\hline
$(5,6,7)$         & $45+8n$           & none                            & \checkmark & $50+8n$             & $8n+41>2$    \\
\hline
$(1,2,3,4)$       & $49+16n$          & none                            & \checkmark & $106+16n$           & $16n-7>2$   \\
\hline
$(1,2,3,6)$       & $65+16n$          & none                            & \checkmark & $106+16n$           & $16n+25>2$   \\
\hline
$(1,2,5,6)$       & $73+16n$          & none                            & \checkmark & $106+16n$           & $16n+41>2$   \\
\hline
$(3,4,5,6)$       & $79+16n$          & $n=2^{t-4}-5$                   & $\times$   & $106+16n$           & $16n+53>2$   \\
\hline
$(1,2,3,4,5)$     & $129+32n$         & none                            & \checkmark & $218+32n$           & $32n+41>2$   \\
\hline
$(1,2,3,4,7)$     & $161+32n$         & none                            & \checkmark & $218+32n$           & $32n+105>2$  \\
\hline
$(3,4,5,6,7)$     & $191+32n$         & $n=2^{t-5}-6$                   & $\times$   & $218+32n$           & $32n+165>2$  \\
\hline
$(1,2,3,4,5,6)$   & $301+64n$         & none                            & \checkmark & $442+64n$           & $64n+161>2$  \\
\hline
$(1,2,3,4,5,6,7)$ & $769+128n$        & none                            & \checkmark & $890+128n$          & $128n+549>2$ \\
\hline
\end{tabular}
\end{center}

Since we are interested in isolating degenerate cases from non-degenerate ones, corresponding to $d\neq 2^t-1$ and $d=2^t-1$ respectively, we consider the equation $d=2^t-1$. For instance, in the first row, corresponding to $J=(1,j)$ with $j=2,4,6$ with $d=1+2j+4n$, as $n>0$, for $d+1=2^t$ the possible value for $t$ is bounded below by $3$. We have
$$d+1=2^t\Rightarrow 2+2j+4n=2^t\Rightarrow (1+j)+2n=2^{t-1}\Rightarrow j\textrm{ is odd}.$$
which is a contradiction as possible values for $j$ are $2,4,6$. Hence, $d+1\neq 2^t$ for all $t$. Similar simple computations, using congruency observations, allows to determine the degenerate cases. The non-degenerate cases are obtained by the possible solution to the equation $d+1=2^t$.

\begin{lmm}\label{eliminate-1}
Consider sequences $J$ in the above table so that $2d+1-\mathrm{top}>2$, $d+1\neq 2^t$ for all $t$, and $d+1\equiv 2\textrm{ mod 4}$. Then, there is no $f\in{_2\pi_{2d}\Omega^8S^{n+8}}$ with $h(f)=(Q_Jx_n)^2$.
\end{lmm}

\begin{proof}
Suppose $J$ is chosen as in the lemma, and there is $f\in{_2\pi_{2d}\Omega^8S^{n+8}}$ with $h(f)=(Q_jx_n)^2$. According to Theorem \ref{detect1}, the mapping
$$(\Omega j_{2^{l_0}}^7)\circ f:S^{2d}\to Q\Sigma^{-1}D_{2^{l_0}}(S^{n+1},7)$$
satisfies $h((\Omega j_{2^{l_0}}^7)\circ f)=(\Sigma^{-1}\overline{Q_Jx_{n+1}})^2$ with $\Sigma^{-1}\overline{Q_Jx_{n+1}}\in H_*\Sigma^{-1}D_{2^{l_0}}(S^{n+1},7)$ and the stable adjoint of $(\Omega j_{2^{l_0}}^7)\circ f$, say $g:S^{2d+1}\to D_{2^{l_0}}(S^{n+1},7)$, is detected by $Sq^{d+1}$ on a $(d+1)$-dimensional class. However, the inequality $2d+1-\mathrm{top}>2$ shows that $H^{2d-1}\Sigma^{-1}D_{2^{l_0}}(S^{n+1},7)\simeq 0$ as well as $H^{2d}\Sigma^{-1}D_{2^{l_0}}(S^{n+1},7)\simeq 0$. Hence, by Lemma \ref{main1}, it is impossible to have $h((\Omega j_{2^{l_0}}^7)\circ f)$ as claimed. This is a contradiction to the existence of such an $f$. This complete the proof.
\end{proof}

The degenerate cases that are not eliminated by Lemma \ref{eliminate-1} are
$$\begin{array}{lcc}
J=(1,j)   &\textrm{ with }&(j=4;n=1),(j=2;n=2,1)\\
J=(1,2,j) &\textrm{ with }&(j=3;n=2,1)
\end{array}$$
with the corresponding $Q_Jx_n$, written in upper indexed operations, respectively equal to
$$\begin{array}{l}
Q^7Q^5x_1,\ Q^7Q^4x_2,\ Q^5Q^3x_1,\\
Q^{17}Q^{9}Q^5x_2,\ Q^{13}Q^7Q^4x_1.
\end{array}$$
Apart from two of the above classes, for the rest of the classes we show that $Q_Jx_n$ is not $A$-annihilated which together with the Cartan formula $Sq^{2t}_*\xi^2=(Sq^t_*\xi)^2$ shows that is not $(Q_Jx_n)^2$ $A$-annihilated, hence not spherical. Applying Nishida relations we have
$$\begin{array}{rcl}
Sq^4_*Q^7Q^5x_1          &=&Q^{5}Q^3x_1+\textrm{other terms}\neq 0,\\
Sq^2_*Q^7Q^4x_2          &=&Q^6Q^3x_2+\textrm{other terms}\neq 0,\\
Sq^4_*Q^{13}Q^7Q^4x_1    &=&Q^{11}Q^6Q^3x_1+\textrm{other terms}\neq 0.
\end{array}$$
It remains to eliminate $(Q_1Q_2x_n)^2=(Q^5Q^3x_1)^2$ from being spherical in $H_*\Omega^8S^{9}$ as well as $(Q^{17}Q^{9}Q^5x_2)^2\in H_*\Omega^8S^{8+2}$.

We have the following which is a more general result.
\begin{lmm}\label{unst-stable-1}
The Hurewicz homomorphism ${_2\pi_{18}}\Omega^lS^{l+1}\to H_{18}\Omega^lS^{l+1}$ is acts trivially.
\end{lmm}

\begin{proof}
Suppose there $l>0$ and $f\in{_2\pi_{18}}\Omega^lS^{l+1}$ with $h(f)\neq 0$. Since the inclusion $E^\infty:\Omega^lS^{l+1}\to QS^1$ induces a monomorphism in homology, then $h(E^\infty f)\neq 0$. Hence, for the adjoint mapping $g:S^{17}\to QS^0$ is also nonzero in homology. On the other hand, recall that ${_2\pi_{17}}QS^0\simeq\Z/2\{\eta\eta^*,\nu\kappa\}$ \cite[Appendix 3]{Ravenel-Greenbook}. But, by results of \cite{Za-ideal}
(see also \cite[Theorem 1.8]{Zare-PEMS}) we have
$$h(\eta\eta^*)=0,\ h(\nu\kappa)=0$$
as $\eta$ and $\eta^*$ as well as $\nu$ and $\kappa$ are not in the same dimension. This contradicts the fact that $h(g)\neq 0$. Hence, we cannot have any $f\in\pi_{18}\Omega^lS^{l+1}$ which maps nontrivially under $h:\pi_{18}\Omega^lS^{l+1}\to H_{18}\Omega^8S^9$, and completes the proof.
\end{proof}

For $l=8$, the above lemma provides us with the required elimination. The reasoning behind Lemma \ref{unst-stable-1} proves a more general fact.

\begin{prp}\label{unst-stable-2}
Suppose $h:{_2\pi_k^s}S^n\to\simeq{_2\pi_k}QS^n\to H_kQS^n$ is trivial. Then, $h:{_2\pi_k}\Omega^lS^{l+n}\to H_k\Omega^lS^{l+n}$ is trivial for all $l>0$. In particular, if $h:{_2\pi_{k-n}^s}\to\simeq{_2\pi_{k-n}}QS^0\to H_kQS^0$ is trivial, then $h:{_2\pi_{k-n+i}}\Omega^{l+n-i}lS^{l+n}\to H_{k-n+i}\Omega^{l+n-i}S^{l+n}$ is trivially for all $l>0$ and $i>0$.
\end{prp}

We can offer various ways to eliminate $(Q^{17}Q^9Q^5x_2)^2\in H_{66}\Omega^8S^{10}$ from being spherical which we sketch below.\\
First, note that if this class is spherical, i.e there exists $f\in{_2\pi_{66}}\Omega^8S^{10}$ with $h(f)=(Q^{17}Q^9Q^5x_2)^2$, as the stabilisation $E^\infty:\Omega^8S^{10}\to QS^2$ induces a monomorphism in homology, then $E^\infty f\in{_2\pi_{66}}QS^2$ also satisfies $h(E^\infty f)=(Q^{17}Q^9Q^5x_2)^2\in H_{66}QS^2$. By adjointing down $E^\infty f$ twice, we obtain an element, say $g\in{_2\pi_{64}}QS^0\simeq{_2\pi_{64}^s}$ with $h(g)\neq 0$ in $H_{64}QS^0$.\\
(1) Appeal to the table of Wellington to show that there is no spherical class in $H_{64}QS^0$ which gives us a contradiction to the existence of $f$. Hence, proving our claim.\\
(2) Appeal to a detailed analysis of ${_2\pi_{64}^s}$ as computed in \cite{KochmanMahowald} to show that $h:{_2\pi_{64}}QS^0\to H_{64}QS^0$ acts trivially, hence eliminating $(Q^{17}Q^9Q^5x_2)^2$ from being spherical in $H_{66}\Omega^8S^{10}$.\\
We leave the rest of the elimination process to the interested reader. Of course there is a kind of uncertainty here if one is to use the second approach suggested above. After work of Isaksen \cite{Isaksen-stablestems} as well as recent computation of Wang and Xu \cite{WangXu-61stem}, providing a proof of ${_2\pi_{61}^s}\simeq0$, one may wonder that whether the computation of Kochmann and Mahowald on ${_2\pi_{64}^s}$ still holds or the application of new methods could result in some corrections to their computations. Therefore, finding a homological way of eliminating $(Q^{17}Q^9Q^5x_2)^2\in H_{66}\Omega^8S^{10}$ from being spherical would be of special interest.

\begin{rmk}
(i) Although, we have applied the lemma in the above situation. However, we wish to use information about $h:{_2\pi_{k-n+i}}\Omega^{l+n-i}S^{l+n}\to H_{k-n+i}\Omega^{l+n-i}S^{l+n}$ for all $i>0$ and $l>0$ to obtain some information on $h:{_2\pi_{k-n+i}}QS^i\to H_{k-n+i}QS^i$. So, the proposition, is in fact in the reverse direction of this paper!\\
(ii) We wished to have more geometric tools of eliminating the last two classes from being spherical. However, using the available information on ${_2\pi_{18}^2}$ and ${_2\pi_{64}^s}$ seemed to provide more economic ways of proving our claims.
\end{rmk}

\subsection{Degenerate cases for $l(J)>1$; $d+1\equiv 0\textrm{ mod }4$}
We begin by collecting the degenerate cases, i.e. $d+1\neq 2^t$ for all $t$, with $d+1\equiv0\mathrm{mod}4$ from the previous table. We have the following cases.
\begin{center}
\begin{tabular}{|c|c|c|c|c|c|c|}
\hline
$J$               &  $Q^Ix_n=Q_Jx$                                    &  $\textrm{cases with } d+1\neq 2^t$ \\
\hline
$(3,j)$           &  $Q^{3+2j+2n}Q^{j+n}x_n$                          &  $(j=6;n\neq 2^{t-2}-4)$            \\
with $j=4,6$      &                                                   &  $(j=4;\textrm{all }n)$             \\
\hline
$(3,4,j)$         & $Q^{7+2j+4n}Q^{4+j+2n}Q^{j+n}x_n$                 &  $(j=5;n\neq 2^{t-3}-4)$            \\
with $j=5,7$      &                                                   &  $(j=7;n\neq 2^{t-3}-5)$            \\
\hline
$(3,6,7)$         & $Q^{23+4n}Q^{13+2n}Q^{7+n}x_n$                    &  all $n$                            \\
\hline
$(3,4,5,6)$       & $Q^{41+8n}Q^{21+4n}Q^{11+2n}Q^{6+n}x_n$           & $n\neq 2^{t-4}-5$                  \\
\hline
$(3,4,5,6,7)$     & $Q^{97+16n}Q^{49+8n}Q^{25+4n}Q^{13+2n}Q^{7+n}x_n$ & $n\neq 2^{t-5}-6$                 \\
\hline
\end{tabular}
\end{center}
We begin with the straightforward cases.
\begin{lmm}
For $J$ being one of the sequences $(3,4,5,6)$ and $(3,4,5,6,7)$ as well as $(3,4,j)$ with $j=5,7$, the class $Q_Jx_n$ is not $A$-annihilated. Consequently, $(Q_Jx_n)^2$ is not spherical.
\end{lmm}

\begin{proof}
By writing in upper indexed operations we have the classes
$$Q^{7+2j+4n}Q^{4+j+2n}Q^{j+n}x_n,\ Q^{41+8n}Q^{21+4n}Q^{11+2n}Q^{6+n}x_n,\ Q^{97+16n}Q^{49+8n}Q^{25+4n}Q^{13+2n}Q^{7+n}x_n$$
to eliminate. We note that
$$7+2j+4n\equiv1\mathrm{mod}4\textrm{ for }j=5,7,\ 41+8n\equiv1\mathrm{mod}4,\ 97+16n\equiv1\mathrm{mod}4$$
which shows that
$${7+2j+4n-2\choose 2}\equiv1\mathrm{mod}2,\ {41+8n-2\choose 2}\equiv1\mathrm{mod}2,\ {97+16n-2\choose 2}\equiv1\mathrm{mod}2.$$
Hence, using Nishida relations we have
$$\begin{array}{rclcl}
Sq^2_*Q^{7+2j+4n}Q^{4+j+2n}Q^{j+n}x_n                 &=& Q^{5+2j+4n}Q^{4+j+2n}Q^{j+n}x_n                 &\neq&0,\\
Sq^2_*Q^{41+8n}Q^{21+4n}Q^{11+2n}Q^{6+n}x_n           &=& Q^{39+8n}Q^{21+4n}Q^{11+2n}Q^{6+n}x_n           &\neq&0,\\
Sq^2_*Q^{97+16n}Q^{49+8n}Q^{25+4n}Q^{13+2n}Q^{7+n}x_n &=& Q^{95+16n}Q^{49+8n}Q^{25+4n}Q^{13+2n}Q^{7+n}x_n &\neq&0.
\end{array}$$
This completes the proof.
\end{proof}

Now, we deal with the two remaining cases separately.
\begin{lmm}
For $J=(3,6,7)$ the class $Q_Jx_n=Q^{23+4n}Q^{13+2n}Q^{7+n}x_n$ is not $A$-annihilated. Consequently, $(Q_Jx_n)^2$ is not spherical.
\end{lmm}

\begin{proof}
If $n$ is odd then $7+n$ is even. In this case, by Nishida relations, we have
$$Sq^4_*Q^{23+4n}Q^{13+2n}Q^{7+n}x_n=Q^{21+4n}Q^{12+2n}Q^{6+n}x_n+\textrm{other terms}$$
where one can compute that other terms are of strictly lower excess. Consequently, if $n$ is odd then $Sq^4_*Q^{23+4n}Q^{13+2n}Q^{7+n}x_n\neq 0$. Hence, for $n$ even, the class $(Q_Jx_n)^2$ is not also $A$-annihilated which means that this class is not also spherical.\\
If $n$ is even, then $13+2n\equiv1\mathrm{mod}4$. In this case, by Nishida relations, we have
$$Sq^4_*Q^{23+4n}Q^{13+2n}Q^{7+n}x_n=Q^{21+4n}Q^{11+2n}Q^{7+n}x_n+\textrm{other terms}\neq 0.$$
This shows that $Q_Jx_n=Q^{23+4n}Q^{13+2n}Q^{7+n}x_n$ is not $A$-annihilated, hence $(Q_Jx_n)^2$ is not. This completes the proof.
\end{proof}

It remains to eliminate the case of $J=(3,j)$ with $Q_Jx_n=Q^{3+j+2n}Q^{j+n}x_n$ where for $j=6$ we require $n\neq 2^{t-2}-4)$, and all possible values for $n$ are allowed if $j=4$. We have the following.
\begin{lmm}
(i) The class $Q^{7+2n}Q^{4+n}x_n$, corresponding to $J=(3,4)$, is not $A$-annihilated. Consequently, $(Q_Jx_n)^2$ is not spherical.\\
(ii) The class $Q^{9+2n}Q^{6+n}x_n$, corresponding to $J=(3,6)$, is not $A$-annihilated. Consequently, $(Q_Jx_n)^2$ is not spherical.
\end{lmm}

\begin{proof}
Note that in both cases, similar to the previous lemma, we need $4+n$ and $6+n$ to be odd which implies $n$ cannot be even; if $n$ is even, we may use $Sq^4_*$ to show that $(Q_Jx_n)^2$ is not $A$-annihilated. Hence, we only deal with the case of $n$ being odd.\\
(i) We divide the proof into two cases: $n\equiv1\mathrm{mod}4$ and $n\equiv3\mathrm{mod}4$.\\
\textbf{Case of} $n\equiv1\mathrm{mod}4$. In this case $7+2n\equiv1\mathrm{mod}4$. By Nishida relations we have
$$Sq^2_*Q^{7+2n}Q^{4+n}x_n=Q^{5+2n}Q^{4+n}x_n\neq 0$$
which shows that $Q^{7+2n}Q^{4+n}x_n$ is not $A$-annihilated, and consequently $(Q^{7+2n}Q^{4+n}x_n)^2$ is not. This completes the proof in this case.\\
\textbf{Case of} $n\equiv3\mathrm{mod}4$. Similarly, $7+2n\equiv1\mathrm{mod}4$. A computation similar to the above case, completes the proof.\\
(ii) The proof is similar to (i) where the only difference is that, here one is to use $Sq^4_*$ and look at the value of $6+n$ modulo $4$. We leave the rest of the computations to the reader.
\end{proof}

\subsection{The non-degenerate cases for $l(J)>1$}
Here, we deal with the non-degenerate cases corresponding to $d=2^t-1$, which are collected in the following table. The second column consists of the classes $Q_Jx_n$, corresponding to a given $J$, when written in upper indexed operations. We show that the following classes cannot give rise to spherical classes $(Q_Jx_n)^2$. Unlike the degenerate cases, where we mostly used geometric tools, here we use algebraic tools. We shows that none of the following classes is $A$-annihilated. Consequently, by the Cartan formula $Sq^{2t}_*\xi^2=(Sq^t_*\xi)^2$, the classes $(Q_Jx_n)^2$ are not $A$-annihilated.
\begin{center}
\begin{tabular}{|c|l|c|c|c|c|}
\hline
$J$                                  &  $Q^Ix_n$                                     & $\textrm{cases with } d+1=2^t$ \\
\hline
$(3,j)$                              &  $Q^{3+2j+2n}Q^{j+n}x_n$                    &  $(j=6,n=2^{t-2}-4)$\\
\hline
$(3,4,j)$                            & $Q^{7+2j+4n}Q^{4+j+2n}Q^{j+n}x_n$           & $(j=5,n=2^{t-3}-4)$\\
with $j=5,7$                         &                                             & $(j=7,n=2^{t-3}-5)$\\
\hline
$(3,4,5,6)$                          & $Q^{41+8n}Q^{21+4n}Q^{11+2n}Q^{6+n}x_n$      & $n=2^{t-4}-5$\\
\hline
$(3,4,5,6,7)$                        & $Q^{97+16n}Q^{49+8n}Q^{25+4n}Q^{13+2n}Q^{7+n}x_n$    & $n=2^{t-5}-6$\\
\hline
\end{tabular}
\end{center}
The cases for with $j+n\equiv2\mathrm{mod}2$ in the above table are immediately eliminated from being spherical. To see this, first note that if $(Q_Jx_n)^2\in H_*\Omega^lS^{n+l}$ is spherical, as the stabilisation $\Omega^lS^{n+l}\to QS^n$ is a multiplicative monomorphism, then $(Q_Jx_n)^2\in H_*QS^n$ is spherical. Now, by computations of Wellington \cite[Theorem 5.3]{Wellington-thesis} it is impossible for $j+n$ to be even. This is equivalent to proving (by induction)that if $I$ consists of an even entry then it is not $A$-annihilated where the induction starts as if $i_1$ is even then $Sq^1_*Q^Ix_n=Q^{i_1-1}Q^{i_2}\cdots Q^{i_s}x_n\neq 0$, and assuming $i_1,\ldots,i_{k-1}$ are odd, one can use $Sq^{2^{k-1}}_*$ to show that $i_k$ also must be odd. Consequently, the classes with $j+n$ being an even number are immediately eliminated. Hence, we have the following remaining cases.
\begin{center}
\begin{tabular}{|c|l|c|c|c|c|}
\hline
$J$                                  &  $Q^Ix_n$                                     & $\textrm{cases with } d+1=2^t$ \\
\hline
$(3,4,j)$                            & $Q^{7+2j+4n}Q^{4+j+2n}Q^{j+n}x_n$           & $(j=5,n=2^{t-3}-4)$\\
\hline
$(3,4,5,6)$                          & $Q^{41+8n}Q^{21+4n}Q^{11+2n}Q^{6+n}x_n$      & $n=2^{t-4}-5$\\
\hline
$(3,4,5,6,7)$                        & $Q^{97+16n}Q^{49+8n}Q^{25+4n}Q^{13+2n}Q^{7+n}x_n$    & $n=2^{t-5}-6$\\
\hline
\end{tabular}
\end{center}
Here, all of these remaining classes fall into the same pattern. By putting the values for $n$ and $j$, we have the following classes to eliminate where we only have computing $i_1$ which is enough for our purpose:
$$\begin{array}{|c|c|}
\hline
Q^Ix_n & \textrm{corresponding to }\\
\hline
Q^{2^{t-2}+1}Q^{4+j+2n}Q^{j+n}x_n & (j=5,n=2^{t-3}-4)\\
\hline
Q^{2^{t-1}+1}Q^{21+4n}Q^{11+2n}Q^{6+n}x_n      &  n=2^{t-4}-5\\
\hline
Q^{2^{t-1}+1}Q^{49+8n}Q^{25+4n}Q^{13+2n}Q^{7+n}x_n    &  n=2^{t-5}-6\\
\hline
\end{array}$$
It is easy to compute by Nishida relations that
$$\begin{array}{lclcc}
Sq^2_*Q^{2^{t-2}+1}Q^{4+j+2n}Q^{j+n}x_n                     &=&  Q^{2^{t-2}-1}Q^{4+j+2n}Q^{j+n}x_n                  &\neq& 0,\\
Sq^2_*Q^{2^{t-1}+1}Q^{21+4n}Q^{11+2n}Q^{6+n}x_n             &=&  Q^{2^{t-1}-1}Q^{21+4n}Q^{11+2n}Q^{6+n}x_n          &\neq& 0,\\
Sq^2_*Q^{2^{t-1}+1}Q^{49+8n}Q^{25+4n}Q^{13+2n}Q^{7+n}x_n    &=&  Q^{2^{t-1}-1}Q^{49+8n}Q^{25+4n}Q^{13+2n}Q^{7+n}x_n &\neq& 0.\\
\end{array}$$
Note that, in the above relations, the result is stated in admissible sequences, so the non-vanishing of the resulting classes is immediate. Hence, these cases are also eliminated from being spherical.

\section{Discussions}
The problem of $p_{2^t-1}^2$ being spherical, implies that there are certain spherical classes in $H_{2^{t+1}-2}QP$ which involve $a_{2^t-1}^2$ as noted in \cite{Eccles-codimension}. Either through the Kahn-Priddy map $QP\to QS^0$ or through certain second James-Hopf map $Q_0S^0\to QP$ one immediately sees that this is related to the existence of certain codimension one immersions. On the other hand, existence of spherical classes $(\sum Q_Jx_n)^2$ in $H_*\Omega^lS^{n+l}$ results in spherical classes in $H_*QS^n$ and relates the problem to the existence of certain codimension $n$ immersions with certain nontrivial structures on their $2^t$-fold point manifolds \cite{KoschorkeSanderson}. The method of Asadi and Eccles \cite{AsadiEccles-determining} together with our observations implies that there exists no such immersion coming from $\Omega^lS^{n+l}$ when $n\geqslant l$ and the dimension is $2^{t+1}-2$ which seems to be a generalised version of Browder's result. The implications of these results to the (unstable) bordism groups of immersions is considered in \cite{Zare-filteredfiniteness}. We leave further detailed investigation on this to a future work.\\

\textbf{Acknowledgment.} I am grateful to Drew Heard for the communication on MathOverFlow regarding ${_2\pi_{64}^s}$ and the current state of knowledge about stable stems in these dimensions who drew my attention to \cite{Isaksen}.

\bibliographystyle{plain}


\begin{thebibliography}{10}

\bibitem{AkhmetevEccles-Browder}
Pyotr M. Akhmetev and Peter J. Eccles.,
{A geometric proof of Browder's theorem on the vanishing of Kervaire invariants.},
{\em Tr. Mat. Inst. Steklova, 225(Solitony Geom. Topol. na Perekrest.)},
{46--51, 1999}.

\bibitem{AsadiEccles-determining},
Mohammad A. Asadi-Golmankhaneh and Peter J. Eccles.,
{Determining the characteristic numbers of self-intersection manifolds.},
{\em J. Lond. Math. Soc., II. Ser.},
{62(1):278--290, 2000.}.



\bibitem{AsadiEccles}
Mohammad A. Asadi-Golmankhaneh and Peter J. Eccles.,
{Double point self-intersection surfaces of immersions.},
{\em Geom. Topol.}
{4:149--170, 2000}.

\bibitem{BE3}
M.G. Barratt and Peter J. Eccles.,
{$\Gamma^+$-structures. III: The stable structure of $\Omega^\infty\Sigma^\infty A$.},
{\em Topology}
{13:199--207, 1974}.

\bibitem{BoardmanSteer}
J. M. Boardman and B. Steer.,
{On Hopf invariants.},
{\em Comment. Math. Helv.},
{42:180--221, 1967}.

\bibitem{Browder}
William Browder.,
{The Kervaire invariant of framed manifolds and its generalization.},
{\em Ann. of Math.}.
{(2) 90:157--186, 1969.}.

\bibitem{CCML}
J. Caruso, F.R. Cohen, J.P. May, and L.R. Taylor.,
{James maps, Segal maps, and the Kahn-Priddy theorem.},
{\em Trans. Am. Math. Soc.},
{281:243--283, 1984.}.

\bibitem{CCKN}
F. R. Cohen, R. L. Cohen, N. J. Kuhn, and Joseph A. Neisendorfer.,
{Bundles over configuration spaces.},
{\em Pacific J. Math.},
{104(1):47--54, 1983.}.

\bibitem{CML}
F. R. Cohen, J. P. May, and L. R. Taylor.,
{Splitting of certain spaces $CX$.},
{\em Math. Proc. Cambridge Philos. Soc.},
{84(3):465--496, 1978.}.

\bibitem{CLM}
Frederick R. Cohen, Thomas J. Lada, and J.Peter May.,
{The homology of iterated loop spaces.},
{\em Lecture Notes in Mathematics. 533. Berlin-Heidelberg-New York: Springer-Verlag. VII},
{1976.}.

\bibitem{Curtis}
Edward B. Curtis.,
{The Dyer-Lashof algebra and the $\Lambda$-algebra.},
{\em Ill. J. Math.},
{19:231--246, 1975.}.

\bibitem{Eccles-codimension}
Peter John Eccles.,
{Codimension one immersions and the Kervaire invariant one problem.},
{\em Math. Proc. Camb. Philos. Soc.},
{90:483--493, 1981.}.

\bibitem{DyerLashof}
Eldon Dyer and R. K. Lashof.,
{Homology of iterated loop spaces.},
{\em Amer. J. Math.},
{84:35--88, 1962.}.


\bibitem{Harper}
John R. Harper.,
{Secondary cohomology operations.},
{\em Providence, RI: American Mathematical Society (AMS)},
{2002.}.

\bibitem{Isaksen-stablestems}
Daniel C. Isaksen.,
{Stable stems.},
{https://arxiv.org/abs/1407.8418.}.

\bibitem{James}
I.M. James.,
{Reduced product spaces.},
{\em Ann. Math.}
{(2), 62:170--197, 1955.}

\bibitem{KochmanMahowald}
Stanley O. Kochman and Mark E. Mahowald.,
{On the computation of stable stems.},
{\em  In The Cech centennial (Boston, MA, 1993), volume 181 of Contemp. Math., pages 299--316. Amer. Math. Soc., Providence, RI,},
{1995.}.

\bibitem{KoschorkeSanderson}
Ulrich Koschorke and Brian Sanderson.,
{Self-intersections and higher Hopf invariants.},
{\em Topology.},
{17:283--290, 1978.}

\bibitem{KudoAraki}
Tatsuji Kudo and Shoro Araki.,
{On $H_*(\Omega^N(S^n); Z_2)$.},
{\em Proc. Japan Acad.},
{32:333--335, 1956.}.

\bibitem{KudoAraki-Hn}
Tatsuji Kudo and Shoro Araki.,
{Topology of Hn-spaces and H-squaring operations.}
{\em Mem. Fac. Sci.Kyusyu Univ. Ser. A.},
{10:85--120, 1956.}.

\bibitem{Kuhngeometry}
Nicholas J. Kuhn.,
{The geometry of the James-Hopf maps.},
{\em Pac. J. Math.},
{102:397--412, 1982.}.

\bibitem{Kuhnhomology}
Nicholas J. Kuhn.,
{The homology of the James-Hopf maps.},
{\em Ill. J. Math.},
{27:315--333, 1983.}.

\bibitem{Madsenthesis}
Ib Madsen.,
{On the action of the Dyer-Lashof algebra in $H_*(G)$ and $H_*(GTop)$.},
{\em PhD thesis, The University of Chicago.},
{1970.}.

\bibitem{Milgram-unstable}
R. James Milgram.,
{Unstable homotopy from the stable point of view.},
{\em Lecture Notes in Mathematics, Vol. 368. Springer-Verlag, Berlin-New York,},
{1974.}.

\bibitem{MosherTangora}
R.E. Mosher and M.C. Tangora.,
{Cohomology operations and applications in homotopy theory.},
{\em Harper's Series in Modern Mathematics. New York-Evanston-London: Harper and Row. X,},
{214 p. ,1968.}.

\bibitem{Ravenel-Greenbook}
Douglas C. Ravenel.,
{Complex cobordism and stable homotopy groups of spheres.}
{\em 2nd ed. Providence, RI: AMS Chelsea Publishing, 2nd ed. edition.},
{2004.}.

\bibitem{Snaith}
V.P. Snaith.,
{A stable decomposition of $\Omega^nS^nX$.},
{\em J. Lond. Math. Soc., II. Ser.},
{7:577--583, 1974.}.

\bibitem{WangXu-61stem}
Guozhen Wang and Zhouli Xu.,
{The triviality of the 61-stem in the stable homotopy groups of spheres.},
{\em Ann. of Math.},
{(2), 186(2):501--580, 2017.}.

\bibitem{Wellington-thesis}
Robert J. Wellington.,
{The $A$-algebra $H_*\Omega^{n+1}\Sigma^{n+1}X$, the Dyer-Lashof algebra, and the $\Lambda$-algebra.},
{\em PhD thesis, The University of Chicago.},
{1977.}

\bibitem{Wellington}
Robert J. Wellington.,
{The unstable Adams spectral sequence for free iterated loop spaces.},
{\em Mem. Am. Math. Soc.},
{36(258):225, 1982.}.

\bibitem{Zare-filteredfiniteness}
Hadi Zare.,
{Filtered finiteness of the image of the unstable hurewicz homomorphism with applications to bordism of immersions.},
{submitted.}.

\bibitem{Zare-PEMS}
Hadi Zare.,
{Freudenthal theorem and spherical classes in $H_*QS^0$.},
{submitted.}.

\bibitem{Za-ideal}
Hadi Zare.,
{On the Hurewicz homomorphism on the extensions of ideals in $\pi_*^s$ and spherical classes in $H_*Q_0S^0$.},
{arXiv:1504.06752.}.

\bibitem{Zare-Els-1}
Hadi Zare.,
{Spherical classes in some finite loop spaces of spheres.},
{\em Topology and its Applications,},
{224--1 18, 2017.}.

\end{thebibliography}

\end{document}